\documentclass[a4paper, 11 pt, twoside]{amsart}
\usepackage{pdfsync}
\usepackage{mathdots}

\usepackage{amsmath,amssymb,amscd,amsfonts,verbatim}
\usepackage{paralist}
\usepackage[mathscr]{eucal}
\usepackage{xypic}
\usepackage{xy}
\usepackage{pifont}
\usepackage[left = 3.5cm, right = 3.5cm, headsep = 6mm,
footskip = 10mm, top = 35mm, bottom = 35mm, footnotesep=5mm, headheight =
2cm]{geometry}
\usepackage{enumerate}

\usepackage{hyperref}

\hypersetup{citecolor=blue,linktocpage}

\newtheorem{thm}{Theorem}[section]
\newtheorem{lem}[thm]{Lemma}

\newtheorem*{thmA}{Theorem A}

\newtheorem{prop}[thm]{Proposition}

\newtheorem{assu-nota}[thm]{Assumption--Notation}
\theoremstyle{definition}

\newtheorem{rem}[thm]{Remark}

\newtheorem{rem-def}[thm]{Remark-definition}

\usepackage{pifont}
\newcommand{\cmark}{\ding{51}}%
\newcommand{\xmark}{\ding{55}}%

\newcommand{\C}{\mathbb C}
\newcommand{\A}{\mathbb A}
\newcommand{\Z}{\mathbb Z}
\newcommand{\Q}{\mathbb Q}

\newcommand{\F}{\mathbb F}

\newcommand{\T}{\mathbb T}
\newcommand{\pp}{\mathbb P}

\newcommand{\Fc}{\mathcal F}

\newcommand{\Sca}{\mathcal S}
\newcommand{\Xca}{\mathcal X}
\newcommand{\Yca}{\mathcal Y}
\newcommand{\Bca}{\mathcal{B}}

\newcommand{\Tc}{\mathcal T}

\newcommand{\kC}{\mathfrak C}
\DeclareMathOperator{\Aut}{Aut}
\DeclareMathOperator{\Def}{Def}

\DeclareMathOperator{\Pic}{Pic}

\DeclareMathOperator{\Hom}{Hom}

\DeclareMathOperator{\Proj}{Proj}

\DeclareMathOperator{\pr}{pr}

\newcommand{\epsi}{\varepsilon}

\newcommand{\fie}{\varphi}

\newcommand{\OO}{\mathcal{O}}
\newcommand{\ol}{\overline}

\newcommand{\inv}{^{-1}}

\def\ge{\geqslant}
\def\le{\leqslant}
\def\geq{\geqslant}
\def\leq{\leqslant}
\numberwithin{equation}{section}

\title[On some boundary divisors]{On some boundary divisors in the moduli spaces of stable Horikawa surfaces with  $K^2=2p_g-3$}
\author{Ciro Ciliberto and Rita Pardini}

\address{Dipartimento di Matematica, Universit\`a di Roma ``Tor Vergata'', Via della 
Ricerca Scientifica, 00177 Roma, Italia}
\email{cilibert@axp.mat.uniroma2.it}

\address{Dipartimento di Matematica, Universit\`a di Pisa, Largo B. Pontecorvo 5, 56127 Pisa, Italia}
\email{rita.pardini@unipi.it}
\begin{document}
\begin{abstract}
We describe  the normal stable surfaces with $K^2=2p_g-3$ and  $p_g\ge 15$ whose only non canonical singularity is a cyclic quotient singularity of type $\frac{1}{4k}(1,2k-1)$ and the corresponding locus $\mathcal  D$  inside the KSBA moduli space of stable surfaces.  More precisely,  we show that: (1) a general point of  any irreducible component of $\mathcal D$ corresponds to a surface with a singularity of type $\frac 14(1,1)$, (2) $\ol{\mathcal D}$ is a divisor contained in the closure of the Gieseker moduli space of canonical models of surfaces with $K^2=2p_g-3$ and intersects all the components of such closure,  and (3) the KSBA  moduli space is smooth at a general point of $\mathcal D$. In addition, we show that $\mathcal D$ has 1 or 2 irreducible components,  depending on the residue class of $p_g$ modulo 4. 
\end{abstract}
\thanks{Acknowledgements: The authors are  members of GNSAGA 
of the Istituto Nazionale di Alta Matematica ``F. Severi''}
\maketitle
 \setcounter{tocdepth}{1}
\tableofcontents
\section{Introduction}
Given positive integers $a,b$,  there is a coarse moduli space $\mathcal M_{a,b}$, whose existence has been proven by Gieseker \cite{Gie}, parametrizing isomorphism classes of canonical models of surfaces of general type with $K^2=a$, $\chi(\OO)=b$. 
The space $\mathcal M_{a,b}$ is quasi--projective and admits a modular ``compactification'', the Koll\'ar--Shepherd-Barron--Alexeev moduli space $\mathcal M_{a,b}^{\rm KSBA}$ of stable surfaces, which is a projective scheme. We point out that,  contrary  to the case of curves,   $\mathcal M_{a,b}$ is open in $\mathcal M_{a,b}^{\rm KSBA}$, but  it is often not dense. 

The first and fundamental step of the construction of the moduli space of stable surfaces was the seminal paper \cite{KSB}, where the notion of stable surface was introduced for the first time. However the development  of the theory of moduli spaces of stable objects, including the case of varieties or pairs of any dimension, required over twenty years and the efforts of many mathematicians, Jan\'os Koll\'ar in the first place. The interested reader can find a complete account of the theory in the recent book \cite{kollar-families} by Koll\'ar. We just mention here that one of the main difficulties was finding the  correct definition of stable families  and that for the case of surfaces the stable families coincide with  the so-called $\Q$--Gorenstein deformations. 

In the past few years there has been an increasing interest  in the explicit  description of $\mathcal M_{a,b}^{\rm KSBA}$,  in particular of the closure $\ol{\mathcal M_{a,b}}$ of  $\mathcal M_{a,b}$ in $\mathcal M_{a,b}^{\rm KSBA}$, for concrete examples.
In most cases (\cite{Ra}, \cite{FPRR}, \cite{Pard}, \cite{pearlstein-etc}, \cite{rollenske-torres}, \cite{2-gorenstein}, \cite{rana-rollenske}, \cite{MNU}) the chosen examples are surfaces lying on the first or the second Noether line, namely surfaces satisfying $K^2=2p_g-4$ or $K^2=2p_g-3$. These surfaces and their moduli spaces have been described in detail by Horikawa (\cite{Ho1}, \cite{Ho2}) and for this reason are often called {\em Horikawa surfaces}. We  say that  a  Horikawa surface is {\em of the first [resp. of the second] kind} if it lies on  the first [resp. second] Noether line. 

Even in the ``smallest'' and most studied case, $K^2=1$, $p_g=2$ (the so-called I--surfaces or $(1,2)$--surfaces), a complete description of  the KSBA moduli space seems completely out of reach. 
One of the possible strategies to get around  this difficulty, used in several  of the above cited papers,  consists in restricting one's attention to the stable surfaces that have  only singularities of class  T (cf. \S \ref{ssec:Tsurf} for the definition). 
 The recent preprint \cite{MNU} contains the classification of stable Horikawa surfaces of the first kind with T--singularities only: in particular it turns out  that  for each value of  $p_g\ge 10$   there is only one family of stable surfaces with T--singularities that might give a divisor inside the closure of the Gieseker moduli space. 
 
Here we consider Horikawa surfaces of the second kind and we focus on the simplest (non canonical) T-singularity, the cyclic quotient singularity $\frac{1}{4}(1, 1)$ and, on   the T--singularities of type $\frac{1}{4\delta}(1,2\delta-1)$ (cf.  \S \ref{ssec:Tsurf} for the notation).  All these singularities have Cartier index 2,  and they all deform to $\frac 14(1, 1)$, which is expected to appear on a codimension 1 subset of  stable surfaces in the closure of the Gieseker moduli space. We take an ``asymptotic'' point of view, meaning that we are interested in the geometry of the locus of surfaces with these singularities for large enough $p_g$.

In this spirit, we  summarize our main results in Theorem A below (i.e., Theorem \ref{thm:smoothD}). Let $n$ be an integer and denote by $\mathcal D_n\subset {\mathcal M^{\rm KSBA}_{2n-1,n+2}}$ the locus of stable Horikawa surfaces with $p_g=n+1$ and $K^2=2p_g-3=2n-1$ whose only non canonical singularity is a singularity of type $\frac{1}{4\delta}(1,2\delta-1)$ for some $\delta\ge1$.

\begin{thmA}\label{thm:intro}
 If $n\ge 14$ then: 
 \begin{enumerate}
\item $\ol{\mathcal D}_n$ is a generically Cartier reduced  divisor contained  in the closure $\ol{\mathcal M_{2n-1,n+2}}$ of the Gieseker moduli space of Horikawa surfaces  with $p_g=n+1$ and $K^2=2p_g-3=2n-1$;
\item $\ol{\mathcal  D}_n$ is irreducible  if $n\equiv 2,3$ modulo $4$ and it has two irreducible components otherwise.   When $\ol{\mathcal M_{2n-1,n+2}}$ has two irreducible components (i.e., when $n\equiv 0$ modulo $4$), then $\ol{\mathcal D}_n$ intersects each of them  in an irreducible divisor.
\item If $X$ is general in a component of $\ol{\mathcal  D}_n$, then the only non canonical singularity of $X$ is of type $\frac 14(1,1)$.
\end{enumerate}
\end{thmA}
Our approach to the description   of $\mathcal D_n$  rests on the observation that the minimal desingularization $f\colon S\to X$ of a stable  surface $X$ corresponding to a point in $\mathcal D_n$ is a smooth minimal Horikawa surface $S$ of the first kind and the  $f$-exceptional divisor $\mathfrak C$ mapping to the T-singularity is a string of smooth rational curves whose self--intersection is determined by $\delta$ (for $\delta=1$ the divisor $\mathfrak C$ is simply a $(-4)$-curve). Conversely by Artin's contractibilitiy theorem, given a surface $S$  with a divisor $\mathfrak C$ as above, one can contract $\mathfrak C$ and all the $(-2)$-curves disjoint from it obtaining a stable Horikawa surface $X$ with a T--singularity of  Cartier index 2. 

So we make a detailed study of the pairs $(S, \mathfrak C)$ (\S \ref{sec:Tchain}) and then (\S \ref{sec:Hor-sing}) we use it to give a stratification of $\mathcal D_n=\sqcup_d \mathcal D_{n,d}$, where $d$ ranges over a suitable finite set of positive integers.  In Section  \ref{sec:explicit} we describe explicitly a $\Q$--Gorenstein smoothing for the surfaces belonging to some of the strata,  showing in particular that all components of $\ol{\mathcal M_{2n-1,n+2}}$ intersect $\mathcal D_n$. We complete the proof of Theorem A by a deformation theoretic argument, based on the  computation of  tangent and obstruction space to the $\Q$-Gorenstein deformations of the general surface $X\in \mathcal D_{n,d}$ for every admissible $d$.

\smallskip

\noindent{\em Acknowledgment:} the authors wish to thank Barbara Fantechi for the  many interesting mathematical conversations  on deformation theory. 
\smallskip

\noindent{\bf Notation and conventions.} We work over the complex numbers.

For $d\ge 0$ we denote by $\F_d$ the Segre--Hirzebruch surface $\Proj(\OO_{\pp^1}(-d)\oplus \OO_{\pp^1})$. We denote by $F$ a fiber of the ruling, by $D$ the tautological section (so  $D^2=-d$,  and $D$ is uniquely determined unless $d=0$) and we set $D_0:=D+dF$. Whenever we write ``$P\in D$'' or ``$P\notin D$'' for a point $P\in \F_d$ we implicitly assume that $d>0$. 

\section{Moduli spaces and T--singularities}\label{sec:mod}

In this section we recall various (mostly) well known facts useful in the rest of the paper. 

\subsection{KSBA moduli spaces}\label{ssec:KSBA}
 In the epochal paper \cite {KSB} the authors introduced a natural compactification of the Gieseker moduli space $\mathcal M_{K^2, \chi}$  of canonical models of surfaces of general type with fixed $K^2$ and $\chi$ (see \cite {Gie}). Due to Alexeev's boundedness result  (see \cite {Ale}), 
the aforementioned compactification is represented by a projective scheme (see \cite {K}, \cite{kollar-families}). The compactification in question is obtained by embedding  $\mathcal M_{K^2, \chi}$ in the larger moduli space of the so--called \emph{stable surfaces}. 

In conclusion we have a projective \emph{KSBA moduli space} $\mathcal M_{K^2, \chi}^{\rm KSBA}$  of stable surfaces with fixed $K^2$ and $\chi$, that contains points corresponding to isomorphism classes of the canonical models of  surfaces of general type with given $K^2$ and $\chi$. 
\medskip

  Not every flat family of stable surfaces $f\colon X \to Z$ induces a morphism to the KSBA moduli space, one has to consider  {\em stable families}. The notion of stable family  is very technical, so we 
refer the reader to \cite{kollar-families} for the precise definitions  and we just recall it here in the case of  families over a reduced  base  (cf. \cite[Def.-Thm.~3.1]{kollar-families}). 
So, a flat  family   $f\colon X\to Z$, with $Z$ reduced,   is said to be stable (or a {\em $\Q$-Gorenstein deformation}) if all the fibers are stable surfaces and $X$ is $\Q$-Gorenstein, i.e.,  there exists an integer $m>0$ such that $mK_{X/Z}$ is a Cartier divisor. Note that if $Z$ is smooth then $K_{X/Z}=K_X-f^*K_Z$ and so $mK_{X/Z}$ is Cartier   if and only if   $K_X$ is.

Koll\'ar has given a very powerful numerical criterion to test stability that reads as follows in  the particular case we are interested in:
\begin{thm}[\cite{kollar-families}, Thm. 5.1]\label{thm:numerical}
Let $f\colon  X\to Z$ be a flat proper  morphism with  $Z$ reduced and irreducible, such that all the fibers $X_z$, $z\in Z$,  are stable surfaces. If the function $z\mapsto K_{X_z}^2$, $z\in Z$, is constant, then $f$ is a stable  family.
\end{thm}
In order to apply the above result to a concrete family $f\colon X\to Z$ one needs  to show that for all $z\in Z$ the fibers $X_z$ are stable surfaces, in particular that they satisfy Serre condition $S_2$. The following  situation will be of interest to us here.  Let $Z$ be a smooth variety,   let $Y\to Z$ be a proper flat morphism   with reduced  fibers and $Y$ normal, and   let $\pi\colon X\to Y$ be a normal double cover such that for all $z\in Z$ no component of the fiber $Y_z$ is  contained in the branch locus of $\pi$, so that the all fibers $X_z$ of $X\to Z$ are reduced. The involution associated with $\pi$ induces a decomposition $\pi_*\OO_X=\OO_Y \oplus\Fc$, where $\Fc$ is a  divisorial sheaf, i.e., a reflexive rank 1 sheaf. The map $\pi$ is flat over a point $y\in Y$  if and only if  $\Fc$ is  invertible at $y$, in particular it is flat over the smooth locus of $Y$.
Given $z\in Z$, the morphism $\pi$ restricts to a double cover $\pi_z\colon X_z\to Y_z$,  but  the variety  $X_z$   is not automatically $S_2$, since it is possible that the restricted sheaf $\Fc|_{Y_z}$ is not  $S_2$. Since a  sufficient condition for the fibers to be $S_2$ is that $X$  be Cohen--Macaulay, 
 the following  sufficient  condition is enough for our purposes.

  \begin{lem}\label{lem:CM}  Let  $\pi\colon X\to Y$ be a double cover of normal varieties. If $Y$  has  quotient singularities then both $Y$  and $X$ are Cohen--Macaulay. \end{lem}
 \begin{proof}
 The statement is local in the \'etale topology (\cite[Thm. 17.5]{matsumura}), so we may assume that all varieties are affine and that $Y=V/H$ for some  smooth variety $V$ and some  finite  group $H$. Hence $Y$ is Cohen--Macaulay by \cite[Exercise 18.14]{eisenbud}. 
 We denote by $p\colon V\to Y$ the quotient map and consider the commutative diagram:
 \begin{equation}
\xymatrix{
\tilde{X}\ar[r] \ar[d]_{\tilde \pi } &  X\ar[d]^{\pi } \\
 V \ar[r]^{p } & Y}
\end{equation}
where $\tilde \pi$ is  the double cover obtained from $\pi$ by  base change and normalization. 
The morphism $\tilde\pi$ is a flat double cover because $V$ is smooth, and therefore $\tilde X$ is Cohen--Macaulay by \cite[Exercise 18.18]{eisenbud}. The group $H$, that acts on $V$, acts also on $\tilde X$ and the  induced map $\tilde X/H\to X$ is a finite birational morphism of normal varieties, so it is an isomorphism and $X$ is  Cohen--Macaulay   by \cite[Exercise 18.14]{eisenbud} again. 
 \end{proof}

\subsection{T--singularities}\label{ssec:Tsurf} 
The simplest singularities that can occur on smoothable   stable surfaces are the so called \emph{T--singularities}. They are either Du Val singularities or cyclic quotient singularities of the form $\frac 1{\delta m^2}(1, \delta ma-1)$ with ${\rm gcd}(m,a)=1$. These singularities admit a $\mathbb Q$--Gorenstein smoothing  (see \cite {LW}) and appear in a natural way on stable degenerations of canonical surfaces (see \cite [Sect. 3]{KSB}).
 The $\mathbb Q$--Gorenstein smoothings of a T--singularity of the form $\frac 1{\delta m^2}(1, \delta ma-1)$ correspond to a  $\delta$--dimensional subset  of its versal deformation space, so 
 the expected codimension in which the singularity appears in the KSBA moduli space is $\delta$, and therefore, if it appears, this happens  in codimension at most $\delta$ (cf. \S \ref{ssec: Tdef}). T--singularities with $\delta=1$ are called \emph{Wahl singularities}, since Wahl has been the first who studied them (see \cite {W}). 
 
The minimal desingularization of a T-singularity is described via its \emph{Hirze-\\
bruch-Jung continued fraction}. Given a T--singularity of type $\frac 1{\delta m^2}(1, \delta ma-1)$, one has the continuous fraction
$$
\frac {\delta ma-1}{\delta m^2}=[e_1,\ldots, e_r]=e_1- \frac 1{e_2 - \frac  1{\ddots- \frac 1{e_r} } }.
$$
Then the exceptional divisor appearing in the minimal resolution of the T--singularity in question is a \emph{chain} of the form $C_1+\cdots +C_r$, where, for any $i\in \{1,\ldots, r\}$, $C_i$ is a $(-e_i)$--curve (i.e., a smooth irreducible curve $C_i$ of genus 0 such that $C_i^2=-e_i$), and $C_i\cdot C_j=0$ unless $|i-j|=1$, in which case $C_i\cdot C_j=1$. Such a chain is also called a \emph{T--chain of type $[e_1,\ldots, e_r]$}. We sometimes use exponential notation for repeated indices in $[e_1,\ldots, e_r]$.

Given a T--chain of type $[e_1,\ldots, e_r]$, then $[2, e _1, \ldots , e_ {r-1}, e_ r+ 1]$ and $[e_ 1 + 1, e_ 2 , \ldots , e_ r , 2]$ are also T-chains and every  T--chain can be obtained in this way starting from  either $[4]$ or $[3, 2^k, 3]$  (see \cite [Sect. 3]{KSB}). From the proof of \cite [Prop. 3.11]{KSB} one sees that $\delta$ is preserved in the recursion. Notice that 
$[4]$ corresponds to the singularity $\frac 14(1,1)$, that has $\delta=1$, whereas 
$[3, 2^k, 3]$ corresponds to the singularity $\frac 1{4(k+2)}(1, 2k+3)$, with $\delta=k+2$. 

Let $S$ be a smooth surface with a T--chain of type $[e_1,\ldots, e_r]$. Contract it to a T--singularity on a surface $S'$. Then (see \cite [Prop. 20]{L}) one has
\begin{equation}\label{eq: K2}
K^2_{S'}=K^2_S+(r-\delta+1).
\end{equation}
From this formula one sees that, in each recursion step described above, the contribution to $K^2$ increases by 1. Moreover it is immediate to check that this contribution is 1 for T-chains of type $[4]$ or $[3, 2^k, 3]$, that are therefore the only T--chains  for which the contribution to  $K^2$ is 1. 
 These are precisely the  T--singularities  with canonical sheaf  of  Cartier index $2$.  For brevity we call  them {\em  2--Gorenstein T--singularities} and call  the corresponding chains {\em 2--Gorenstein T--chains}.
 
Given a normal  surface  singularity $P\in U$ with minimal resolution   $p\colon V\to U$ and  exceptional divisor $\mathfrak C=\sum_{i=1}^rC_i\subset V$, if $K_U$ is a $\Q$--Cartier divisor  then   there are rational numbers $a_1,\dots, a_r$ such that  $K_V=p^*K_U+\sum_{i=1}^r a_iC_i$. Since the intersection matrix $(C_i\cdot C_j)$ is negative definite, the $a_i$ are uniquely determined by requiring
  \begin{equation} \label{eq:adj} 2p_a(C_j)-2=C_j^2+K_V\cdot C_j = C_j^2+C_j\cdot \sum_{i=1}^ra_iC_i,
  \end{equation}
  for $j=1, \dots,  r$. If the singularity is of   type $\frac 1{4(k+2)}(1, 2k+3)$, with $k\ge -1$, then it is immediate to check that \eqref{eq:adj} holds for $a_1=\dots=a_r=-\frac 12$, hence $K_V=p^*K_U-\frac 12 \mathfrak C$.
 
 \subsection{Stable 2--Gorenstein T--singular Horikawa surfaces of the second kind}
 
For brevity we call {\em T--singular surface} a surface whose only non canonical singularity is  a T--singularity. Here we collect some basic facts on T-singular Horikawa surfaces of the second kind that will be useful later on.  
\begin{prop} \label{prop:XS}
\begin{enumerate}
\item Let $X$ be a 2--Gorenstein stable T--singular surface and let  $S\to X$ be the minimal resolution. 
If $K^2_X>1$ and $p_g(X)>0$, then $S$ is a minimal surface of general type with $p_g(S)=p_g(X)$ and $K^2_S=K^2_X-1$. 

\item Conversely, given a minimal surface $S$ of general type with a 2-Gorenstein T--chain $\kC$, one can contract $\kC$ and all the $(-2)$--curves of $S$ disjoint from it, obtaining a stable surface  $X$ with $K^2_X=K^2_S+1$ and $p_g(X)=p_g(S)$.
\end{enumerate}
\end{prop}
\begin{proof}
(i) The invariants of $S$ are $K^2_S=K^2_X-1>0$, by \eqref{eq: K2},  and $p_g(X)=p_g(S)>0$ because quotient singularities are rational. So $S$ is of general type. If we denote  by  $W$ the minimal model of $S$,   \cite[Thm.~1.1]{Ra} gives $K^2_X>K^2_W\ge K^2_S=K^2_X-1$, namely $K^2_S=K^2_W$  and $S=W$ is minimal.
\medskip

(ii) By \cite[Thm.~2.3]{artin} there is a  morphism $f\colon S\to X$, contracting to normal singularities precisely  $\kC$ and the $(-2)$--curves of $S$ disjoint from it. Since  $f^*K_X=K_S+\frac 12 \kC$ (cf. \S \ref{ssec:Tsurf}),  in order to prove that $X$ is stable it is enough to prove  that  $f^*K_X$ is nef and big  and  that an irreducible curve has zero intersection with   $f^*K_X$  if and only if  it  is $f$--exceptional. 
Note first of all that $f^*K_X$ is big since it is the sum of $K_S$ and of an effective divisor.
Next  let $\theta$ be an irreducible curve of $S$.  If  $\theta$ is a component of $\mathfrak C$, then  one can check directly that $\theta \cdot (K_S+\frac 12 \kC)=0$ 
 using the fact that $\mathfrak C$ is a string of type $[4]$ or $[3, 2^k,3]$. If $\theta$ is not a component of $\mathfrak C$,  then both $\theta \cdot K_S $ and $\theta \cdot\mathfrak C$ are non-negative, since $K_S$ is nef;  in addition if $\theta\cdot (K_S+\frac 12 \kC)=0$, then $\theta \cdot K_S =\theta \cdot\mathfrak C=0$, namely $\theta$ is a $(-2)$--curve disjoint from $\mathfrak C$.
\end{proof}
The next result  will be needed in \S \ref{sec:Hor-sing}.
 \begin{lem}\label{lem:can-ring}
  Let $S$ be a  minimal surface of general type with a  2--Gorenstein  T--chain  $\kC$  and let $X$ be the    stable  surface  obtained by contracting $\kC$ and the $(-2)$--curves disjoint from $\kC$. Set $L:=2K_S+\kC$.  Then:
 \begin{enumerate} 
 \item  $h^0(mL)= \chi(mL)$ for all $m\in \mathbb N_{>0}$
 \item $X\cong \Proj(\oplus_{m\ge 0}H^0(mL))$
   \item the system $|L|$ is free
 \end{enumerate}
 \end{lem}
 \begin{proof} Let $f\colon S\to X$ be the contraction morphism. Since $K_X$ is ample, $L=f^*(2K_X)$ is nef and  big.  The line bundle $mL$ is the adjoint of    
 $(2m-1)K_S+m\kC=(2m-1)(K_S+\frac 12 \kC)+\frac 12 \kC$, so  (i) follows  by Kawamata--Viehweg vanishing. Since T--singularities are rational, the pull-back map $f^*\colon H^0(2mK_X)\to H^0(mL)$ is an isomorphism for all $m\ge 0$ and  therefore $X= \Proj\left(\oplus_{m\ge 0}H^0(2mK_X)\right)\cong  \Proj\left(\oplus_{m\ge 0}H^0(mL)\right)$ and (ii) is proven.
 
 To prove  (iii) note that $L|_{\kC}$ is the trivial sheaf, since it has degree 0 on all components of $\kC$ and $\kC$ is 1--connected. Since $|2K_S|$ is free, the exact sequence 
 $$0\to H^0(2K_S)\to H^0(L)\to H^0(L|_{\kC})\to H^1(2K_S)=0$$
shows that $|L|$ is also free. 
  \end{proof}

\begin{rem}\label{rem: XS}
By Proposition \ref{prop:XS}, (i)  there is no  stable Horikawa surface of the first  kind  (i.e., with $K^2=2p_g-4$)  whose only non canonical singularity is a 2--Gorenstein T--singularity, since its minimal desingularization  would  be a minimal surface of general type (cf. Proposition \ref{prop:XS}) contradicting Noether inequality.  In addition, the stable Horikawa surfaces of the second  kind  (i.e., with $K^2=2p_g-3$)  with $K^2>1$ and  a 2--Gorenstein T--singularity are  naturally in bijection with classical Horikawa surfaces of the first  kind containing a 2--Gorenstein T--chain. 
\end{rem}

\subsection{First order  deformations of T--singular stable  surfaces}\label{ssec: Tdef}
We briefly recall some  facts on infinitesimal  $\Q$-Gorenstein deformations  (called KSB deformations in \cite[\S~6.6]{kollar-families}) of normal stable surfaces with at most T-singularities that we will need later. A nice account of the topic can be found in \cite[\S 3, \S 4]{hacking}, while \cite[\S 6.6]{kollar-families} contains an exhaustive treatment of the local deformation theory of cyclic quotient singularities.
\medskip

Let $X$ be a normal stable surface whose singular points  are T-singularities. 
Denote by $\mathbb T_{QG}^1$ the space of first order $\Q$-Gorenstein deformations of $X$,  by $\T^2_{QG}$ the corresponding obstruction space and by  $\Tc^1_{QG}$ the sheaf  of   first order local $\Q$-Gorenstein deformations, which is a coherent sheaf supported on the singularities of $X$. So $H^i(\Tc^1_{QG})=0$ for $i>0$.  In addition the local obstruction sheaf $\Tc^2_{QG}$ is known to be zero for T-singularities and there is 
a  local-to-global exact sequence:
\begin{equation}\label{eq: local-to-global}
0\to H^1(T_X)\to\T^1_{QG}\to H^0(\Tc^1_{QG})\to H^2(T_X)\to \T^2_{QG}\to 0,
\end{equation}
where $T_X:=\Hom(\Omega^1_X,\OO_X)$ is the \emph{tangent sheaf} of $X$. 

To compute $\T^1_{QG}$ and $\T^2_{QG}$ in concrete cases we need to understand the remaining terms in \eqref{eq: local-to-global}. 
First we consider  the cohomology of $T_X$. The computation can be reduced to a computation on a smooth surface thanks to the following result, whose proof is very similar to that of \cite[Thm.~2]{lee-park}:

\begin{prop}\label{prop:Tlog}
 Let $X$ be a normal surface with T-singularities only, 
let $f\colon S\to X$ be the minimal resolution and let $E$ be the reduced exceptional divisor. Then
 for $i=0,1,2$ there is an isomorphism $H^i(T_X)\cong H^i(T_S(-\log E))$.
\end{prop}
\begin{proof} Note that $E$ is a simple normal crossing divisor, hence $T_S(-\log E)$ is a locally free sheaf and, 	
 writing $E=E_1+\dots +E_r$ with $E_i$ irreducible, we have a short exact sequence 
$$0\to T_S(-\log E)\to T_S\to \oplus_i\OO_{E_i}(E_i)\to 0.$$
Since $E_i^2<0$ for every $i$, pushing forward to $X$ the above sequence induces an isomorphism $f_*T_S(-\log E)\to f_*T_S$. On the other hand by \cite[(1.0)]{burns-wahl} there is a natural isomorphism $f_*T_S\to T_X$ and therefore we obtain  a natural isomorphism $f_*T_S(-\log E)\to T_X$. By \cite[Lem.~1]{lee-park} the higher derived functors $R^jf_*T_S(-\log E)$ vanish for $j>0$, and so $H^i(T_X)=H^i(f_*T_S(-\log E))$ is isomorphic to $H^i(T_S(-\log E))$ for all $i$.
\end{proof}
We describe local deformations and $\Tc^1_{QG}$ only in the two cases that are of interest to us here.\smallskip

 \underline{Singularity of type $A_1$.} In local coordinates this singularity has equation $z^{2}-xy=0$. The miniversal family $\mathcal X\to \A^1$  is the hypersurface $z^{2}-xy+t=0$, where $t\in \A^{1}$. So all deformations of this singularity are Gorenstein and the stalk of $\Tc^1_{QG}$ at the singular point is a $1$-dimensional complex vector space. Note that the involution $\tau$ defined by $(x,y,z)\mapsto (x,y,-z)$ acts on $A_1$ and on $X$, inducing the identity on $\A^1$. \smallskip
 
\underline{$\frac 14(1,1)$--singularity.} It is easy to check that  a T--singularity of type $\frac 14(1,1)$,  is the quotient of the $A_1$ singularity $z^{2}-xy=0$ by the involution $\sigma\colon (x,y,z)\mapsto(-x,-y,-z)$. The quotient map is the so-called \emph{index cover} of the T--singularity. The local $\Q$--Gorenstein deformations of the T--singularity are induced by the deformations of the index cover that are $\sigma$-invariant, namely by those of the form $z^{2}-xy+t=0$. So the $\Q$--Gorenstein deformations correspond to a $1$-dimensional subspace of the miniversal deformation and  the stalk of $\Tc^1_{QG}$ at the singular point is a complex vector space of dimension $1$.
\smallskip

The following observation  will be useful in \S \ref{sec:infinitesimal}.
\begin{rem}\label{rem:sing-inv}
Let $G$ be the group of automorphisms of the $A_1$-singularity $z^{2}-xy=0$ generated by the  involutions $\sigma$ and  $\tau$ defined above. Taking the quotient by $G$ one gets again an $A_1$-singularity $w^{2}-uv=0$, where the quotient map is given by $u=x^2,v=y^2, w=z^2$ and is branched on $w=0$. We have seen that taking the quotient by $\sigma$ only we get the T--singularity of type $\frac 14(1,1)$. So, if we denote by $\bar \tau$ the involution of the T--singularity induced by $\tau$, taking the quotient by $\bar \tau$ represents the  singularity $\frac 14(1,1)$ as a double cover of $w^{2}-uv=0$ branched on $w=0$. The above explicit description of the $\Q$--Gorenstein smoothings shows that all $\Q$-Gorenstein deformations of the singularity $\frac 14(1,1)$ are $\bar\tau$--invariant and, as a consequence, $\bar\tau$ acts trivially on $\Tc^1_{QG}$.
\end{rem}

\section{Horikawa surfaces}\label{sec:gen}

In this section we  recall some of Horikawa's results in \cite {Ho1, Ho2} on minimal surfaces of general type with $K^2=2p_g-4$ [resp.  $K^2=2p_g-3$], that,  as we said in the Introduction, we call Horikawa surfaces of the first [resp. of the second]  kind.
 We also define a stratification of the  moduli space of the canonical models of the Horikawa surfaces of the second kind.

\subsection{Horikawa surfaces of the  first kind}\label{ssec:one}

Let us start with Horikawa surfaces of the  first kind. For these surfaces we have  $p_g\geq 3$ and the irregularity is $q=0$.  Throughout the paper we set $p_g=n+1$. 

The basic information about these surfaces (see \cite [Lemma 1.1]{Ho1}) is that the canonical system $|K|$ is base point free, not composed with a pencil and the canonical map $\varphi_K: S\longrightarrow \pp^n$ is a 2:1 morphism onto a surface of minimal degree $n-1$ in $\pp^n$. Hence $S$ has an involution $\iota: S\longrightarrow S$ (called the \emph{canonical involution}) with which the canonical map $\varphi_K$ is composed. 

The surfaces of minimal degree in $\pp^n$ (with $n\geq 2$) are (see \cite {Na}):\\ \begin{enumerate}
\item [(i)] the plane $\pp^2$ for $n=2$;
\item [(ii)] the 2--Veronese surface of the plane for $n=5$;
\item [(iii)] rational normal scrolls of degree $n-1$ in $\pp^n$ for $n\geq 3$; these are in turn: 
\begin{itemize}
\item [(iii.a)] smooth rational normal scrolls, i.e., the  isomorphic image $S_{d,n}$ in $\pp^n$ of the  Hirzebruch surface $\F_d$ via the linear system $|D + \frac {n - 1 + d}2 F|$ \footnote{Recall the notation introduced at the end of the Introduction. Note that the ruling of $S_{d,n}$ is uniquely determined unless $d=0, n=3$.},  $n - 1 + d$ is even and  $d\leq n-3$;
\item [(iii.b)] cones over rational normal curves,  that are the images  $S_{n-1,n}$ in $\pp^n$ of a Hirzebruch surface $\F_{n-1}$ via the linear system $|D + (n - 1) F|$.
\end{itemize}
\end{enumerate}

Note that in case (iii.a) the curve $D$ on $S_{d,n}$ has degree
$$
D\cdot \Big (D + \frac {n - 1 + d}2 F\Big )=\frac {n-d-1}2\geq 1,
$$
and the equality holds if and only if $d=n-3$.

It is a result by Horikawa (see \cite [Thm. 1.6, (iv)]{Ho1}) that the image of a Horikawa surface of the   first kind can be a cone 
 or a surface of type (i) or (ii)
only if $p_g\leq 6$. Since in this paper we will not be interested in Horikawa surfaces with limited invariants, 
 we will focus on Horikawa surfaces of the  first kind whose canonical image is a smooth rational normal scroll $S_{d,n}$ with $d\leq n-3$.   Horikawa defines them to be  \emph{of type $(d)$} if the canonical image is of type (iii.a) above.  For these surfaces Horikawa gives  the following description:

\begin{thm}\label{thm:Ho1a} 
Let $S$ be a Horikawa surfaces of the  first kind with $p_g=n+1\geq 4$ of type $(d)$. Then:
\begin{enumerate}
\item $n -  d$ is odd and  $n\geq \max\{d+3,2d-3\}$;
\item the canonical model of $S$  is a double cover of a Hirzebruch surface $\bar S\to \F_d$ branched along a reduced divisor $B$ in the linear system $|6D+(n+3+3d)F|$ with at most irrelevant singularities.
\end{enumerate}
\end{thm}
\begin{rem}
The branch divisor $B$ is in general irreducible if $D\cdot B=n+3-3d\ge 0$. If $n+3-3d< 0$ then $D$ splits off $B$ with multiplicity 1 (recall that $n\ge 2d-3$ by assumption). In this case $B$ has in genral $\nu:=n+3-2d$ nodes occurring at the intersection of $D$ with $B-D$.
\end{rem}

We denote by  $\mathcal M_{n}^{{\rm Hor,1}}$ the image in the Gieseker moduli space  $\mathcal M_{2n-2, n+2}$ of the Horikawa surfaces  of type (d) with $p_g=n+1$.  Horikawa proves the following:

\begin{thm}\label{thm:Ho1b}
\begin{enumerate}
\item $\mathcal M_{n}^{{\rm Hor,1}}$ is a union of irreducible components of  \\ $\mathcal M_{2n-2, n+2}$ and for $n\ge 6$ it is equal to $\mathcal M_{2n-2, n+2}$;
\item if $K^2=2n-2$ is not divisible by 8, then $\mathcal M_{n}^{{\rm Hor,1}}$ is irreducible of dimension $7n+21$;\item if $K^2=2n-2=8k$ and $k\ge 2$, then $\mathcal M_{n}^{{\rm Hor,1}}$ has two disjoint irreducible  components, both of dimension $28(k+1)=7n+21$:  one, $\mathcal M_{4k+1}^{\rm Hor,1a}$,
 containing the  points corresponding to the surfaces of type $(d)$ with $d$ even and $0\leq d\leq 2k$, the other, $\mathcal M_{4k+1}^{\rm Hor,1b}$, containing the points corresponding to the surfaces of type $(2k+2)$. 
\end{enumerate}
\end{thm}

\subsection{Horikawa surfaces of the second  kind}\label{sec:2p_g-3}

Let us turn now to Horikawa surfaces of the second  kind. We have $p_g\geq 2$ and the irregularity is $q=0$. The cases $p_g=2$ and $p_g=4$ have been treated in detail in \cite {Pard, Ra}. If $p_g=n+1\geq 5$, Horikawa proves that the canonical system $|K|$ has a unique simple base point and the canonical map $\varphi_K: S\dasharrow \pp^n$ is a degree 2 map to  a surface of minimal degree $n-1$ containing a line, the image of the exceptional divisor over the base point of $|K|$. Hence this image cannot be the Veronese surface in $\pp^5$ so it is a rational normal scroll.   By imitating the proof of  \cite [Thm. 1.6, (iv)]{Ho1}, one sees that the image of a Horikawa surface of the second   kind can be a cone only if $p_g\leq 7$. Since we will not be interested in Horikawa surfaces with limited invariants, 
 we will focus on Horikawa surfaces of the second  kind whose canonical image is a smooth rational normal scroll  $S_{d,n}$ with $d\leq n-3$. Horikawa defines them to be \emph{of type $(d)$}. For these surfaces Horikawa proves the following:
\begin{thm}\label{thm:ho2.3} Let $S$ be a  Horikawa surface of the second  kind with $p_g=n+1 \geq 7$  of type $(d)$. Then:
\begin{enumerate}
\item $n -  d$ is odd and  $n\geq \max\{d+3,2d-2\}$;
\item the canonical map of $S$ is birationally  a double cover of  $\F_d$ branched along a  curve $B\in |6D+(n+5+3d)F|$ having two (proper or infinitely near) 4--uple points  along a line $F_0\in |F|$, that is a component of $B$, and irrelevant singularities elsewhere. 
\end{enumerate}
\end{thm}
 For fixed $n$, denote by   $\mathcal M_{n}^{{\rm Hor,2}}$ the subset of the moduli space of surfaces of general type consisting the Horikawa surfaces of the second   kind of type (d) with $p_g=n+1$.
\begin{thm}\label{thm:Ho23b}
\begin{enumerate}
\item $\mathcal M_{n}^{{\rm Hor,2}}$ is a union of irreducible components of the moduli space of surfaces of general type and for $n\ge 7$ it contains all the Horikawa surfaces of the second   kind with $p_g=n+1$;
\item if $n$ is not divisible by 4, then $\mathcal M_{n}^{{\rm Hor,2}}$ is irreducible of dimension $7n+19$
\item if $n=4k$,  with $k\geq 2$,  then $\mathcal M_{n}^{{\rm Hor,2}}$ has two irreducible  components, both of dimension $28k+19=7n+19$: one, $\mathcal M_{4k+1}^{\rm Hor,2a}$, containing the  points corresponding to the surfaces of type $(d)$ with $d$ odd and $0\leq d\leq 2k-1$, the other, $\mathcal M_{4k+1}^{\rm Hor,2b}$, containing points corresponding to the surfaces of type $(2k+2)$. 
\end{enumerate}
\end{thm}
Thus in  $\mathcal M_{n}^{{\rm Hor,2}}$ there are the points corresponding to the surfaces of a given type $(d)$, with  $n\geq\max\{d+3, 2d-2\}$ and $n-d$ odd. These form an irreducible constructible set  $\mathcal M_{n,d}^{{\rm Hor,2}}$  that we call the \emph{stratum of type $(d)$}. A surface corresponding to a general point in $\mathcal M_{n,d}^{{\rm Hor,2}}$ will be 
called \emph{general of type (d)}. 
The moduli space $\mathcal M_{n}^{{\rm Hor,2}}$ is the disjoint union  of the strata of type (d) for $0\le d\le \min\{n-3, \frac{n+2}{2}\}$.

Note that in the above setup we can never have $n+4=3d$, because $n-d$ is odd by assumption. 

We now determine the dimensions of the strata. 
\begin{prop}\label{prop:countb}  Let $\eta:=\max\{0, 3d-n-4\}$. Then  $\eta\le d-2$ and
\begin{enumerate}
\item for $n$ odd and $d= 0$ (hence $\eta=0$),   $\mathcal M_{n,0}^{{\rm Hor,2}}$ has dimension $7n+19$  and is dense  in $\mathcal M_{n}^{{\rm Hor,2}}$;
\item  if $d\geq 1$ and $\eta= 0$, then $\mathcal M_{n,d}^{{\rm Hor,2}}$ has dimension $7n+20-d$ and is dense in $\mathcal M_{n}^{{\rm Hor,2}}$ if and only if $d=1$ (and $n$ is even);
\item  if $\eta>0$,  then  $\mathcal M_{n,d}^{{\rm Hor,2}}$ has dimension $7n+21+\eta-d\le 7n+19$ and equality holds if and only if $\eta=d-2$, i.e.  for  $\mathcal M_{4k+1,2k+2}^{\rm Hor,2}=\mathcal M_{4k+1}^{\rm Hor,2b}$. 

\end{enumerate}
 In cases (i) and (ii) (or equivalently if $\eta=0$) the branch curve $B$ of the surface corresponding to the general point in $\mathcal M_{n,d}^{{\rm Hor,2}}$ is of the form $B=F_0+B'$, with $F_0\in |F|$ and $B'\in |6D+(n+4+3d)F|$  irreducible,  and  $B'$ has two ordinary triple points at two distinct points $P_1,P_2$ of $F_0$ 
 and there is 
 no other singularity of $B$. 
 
In case   (iii) (or equivalently  if $\eta>0$),   
the  branch curve $B$ of the surface corresponding to the general point in $\mathcal M_{n,d}^{{\rm Hor,2}}$ is of the form $B=F_0+D+B'$, with $F_0\in |F|$ and $B'\in |5D+(n+4+3d)F|$, and $B'$ has a node at the intersection point $P_1$ of $F_0$ and $D$,   an ordinary triple point at a point $P_2\in F_0$ distinct from $P_1$, 
and intersects $D$ transversally at $n+2-2d$ points off $P_1$ (that are nodes for $B$) and has no other singularity.
\end{prop}
\begin{proof}[Sketch of proof] 
The branch curve $B$ contains a line $F_0\in |F|$. The residual curve sits in the linear system $|6D+(n+4+3d)F|$. We compute 
$$
(6D+(n+4+3d)F)\cdot D=n+4-3d.
$$
So, if $\eta=0$ (i.e., in cases (i) and (ii)) the curve $D$ does not split off $|6D+(n+4+3d)F|$. One has
$$
\dim(|6D+(n+4+3d)F|)=7n+34.
$$
To impose two triple points to $|6D+(n+4+3d)F|$ along $F_0$ is $10$ conditions, and adding one parameter for the variation of $F_0$ in $|F|$, we get $7n+25$ parameters. Finally, subtracting the dimension of the automorphism group of $\F_d$ (i.e., 6 if $d=0$ and $d+5$ otherwise), one gets the results in (i) and (ii).

Next assume $3d>n+4$, namely $\eta\geq 1$ (case (iii)). Note that, being $n\geq 2d-2$, one has $\eta\leq d-2$. In this case the branch curve $B$ contains a line $F_0\in |F|$ and the curve $D$. The residual curve sits in $|5D+(n+4+3d)F|$,  which is a linear system of dimension $7n+33+\eta$.

Next we have to impose to $|5D+(n+4+3d)F|$ a double point at the intersection point of $F_0$ with $D$ and a triple point along $F_0$, and let $F_0$ move, which is 7 conditions. Finally we have to subtract the dimension of the automorphism group of $\F_d$ that is $d+5$. Eventually we find 
$$
7n+33+\eta-(d+12)=7n+21+\eta-d\leq 7n+19
$$
parameters, as wanted. The number of parameters is $7n+19$ if and only if $\eta= d-2$, i.e., if and only if $n+2=2d$. Since $n-d$ is odd, $d$ has to be odd, so we can set $d=2k+1$ (with $k\geq 1$) and accordingly $n=4k$ as wanted.

The final assertions are easy to prove and we can leave the standard details to the reader.   \end{proof} 

\begin{rem}\label{rem:due} The surface corresponding to the general point of $\mathcal M_{n,d}^{{\rm Hor,2}}$ is  smooth  if we are in cases (i)--(ii) of Proposition \ref {prop:countb}, whereas if we are in case (iii) it has exactly $n + 2 - 2d=d-2-\eta$ double points of type $A_1$ if we are in case (iii) of Proposition \ref {prop:countb}.  So for $3d>n+4$ the codimension of $\mathcal M_{n,d}^{{\rm Hor,2}}$ in $\mathcal M_{n}^{{\rm Hor, 2}}$ is equal to the number of $A_1$ points of the  general surface of $\mathcal M_{n}^{{\rm Hor,2}}$.

\end{rem}

\begin{rem}\label{rem:liup} It is useful, for our later purposes, to give an alternative  description of surfaces corresponding to points in $\mathcal M_{n,d}^{{\rm Hor,2}}$. 
Any such a surface is birationally a double cover $\phi: X\longrightarrow \F_d$ branched along a reduced curve $B\in |6D+(n+5+3d)F|$ (with $n\geq\max\{d+3, 2d-2\}$ and $n-d$ odd) having irrelevant singularities besides two (proper or infinitely near) 4--uple points $P_1,P_2$ along a ruling  $F_0\in |F|$ that is a component of $B$. To make things easier, let us suppose that the points $P_1,P_2$ are distinct, what corresponds to a generality assumption in $\mathcal M_{n,d}^{{\rm Hor,2}}$. 

Let us consider the blow-up $\varphi\colon \tilde \F_d \longrightarrow \F_d$ of $\F_d$ at the points $P_1,P_2$. We will denote by $\Phi$ the strict transform of $F_0$ on $\tilde \F_d$, that verifies $\Phi^2=-2$.  Consider the commutative diagram 
\begin{equation}
\xymatrix{
X''\ar[r] \ar[d]^{f } &  X'\ar[r] \ar[d]^{\phi' } & X \ar[d]^{\phi }\\
\tilde \F_d \ar[r]^{{\rm id} } & \tilde \F_d \ar[r]^{\varphi }&  \F_d}
\end{equation}
where the  rightmost  square is cartesian and the morphism $X''\longrightarrow X'$ is the normalization. Let us look at the double cover $f\colon X''\longrightarrow \tilde \F_d $. Its branch curve consists of the curve $B_0$, the strict transform of $B$, and $B_0=\Phi+B_0'$, with $B'_0\cap \Phi=\emptyset$. Moreover  $B_0'\in |\varphi^*(6D+(n+3d+4)F)-3E_1-3E_2|$  has only irrelevant singularities and it contains the strict transform $D'$ of the curve $D$  whenever $3d>n+4$.  Since $\Phi$ is contained in the branch locus of $f$ we have $f^*(\Phi)=2E$, where $E$ is a $(-1)$-curve.

 Let us now consider the commutative diagram 
\[\xymatrix{
X''\ar[r] \ar[d]^{f } &  \bar S \ar[d]^{g} \\
\tilde \F_d  \ar[r]^{\psi} & Z_d}
\] 
where  $X''\to \bar S$  is the contraction of $E$  to a smooth point $p$ and  $\psi\colon \tilde \F_d  \longrightarrow Z_d$ is the contraction of $\Phi$ to a singular point $x=g(p)$ of type $A_1$. Then the morphism  $g\colon  \bar S \longrightarrow Z_d$ is the double cover branched at the point $x$ and along the image $\mathfrak B$ of the curve $B'_0$.  Notice that   $\mathfrak B\cong B'_0$ since $B'_0$ and $\Phi$ are disjoint. Notice also that $\bar S$  can have at most canonical  singularities corresponding to the irrelevant singularities of $\mathfrak B$. Since the morphism $g$ is finite, the Hurwitz formula  gives $2K_{\bar S}= g^*(2K_{Z_d}+\mathfrak B)$. The pull back of $2K_{Z_d}+\mathfrak B$ to $\tilde{\F}_d$  is equal to  $\varphi^*(2D+(n+d)F)-E_1-E_2=\varphi^*(2D_0+(n-d)F)-E_1-E_2$. Since $n-d\ge 3$ by assumption, we may apply Lemma \ref{lem:ample}  and conclude that $\varphi^*(2D+(n+d)F)-E_1-E_2$ is nef and that $\Phi$ is the the only curve that has zero intersection with it. Therefore  $2K_{Z_d}+\mathfrak B$ is ample.

So $K_{\bar S}$ is ample and $\bar S$ is the canonical model of a   Horikawa surface of the second  kind $S$. The quotient of $\bar S$ by the canonical involution $\iota$ is the surface $\tilde\F_d$ and the cover $\bar S\to Z_d$ is branched on $\mathfrak B$ and on the singular point $x$.

Conversely, a double cover as above is a Horikawa surface of the second kind of type (d).
\end{rem}

\section{Minimal surfaces with $K^2=2p_g-4$ containing  a 2--Gorenstein T--chain}\label{sec:Tchain}
 In this section we will classify Horikawa surfaces of the  first kind with $p_g=n+1\geq 4$ of type $(d)$, described in Theorem \ref {thm:Ho1a},
possessing in addition a 2--Gorenstein T--chain  (cf. \S \ref{ssec:Tsurf}). This classification will  be used in \S \ref{sec:Hor-sing} to describe Horikawa surfaces of the second kind with a 2-Gorenstein T-singularity.

Let $\mathfrak C$ be a 2-Gorenstein T--chain. If $\kC$  is of type $[4]$ we denote by  $C$ its only curve, if it is of type [3,3] we write  $\kC=C_0+C_1$ and if it is of type $[3,2^k, 3]$ with $k\ge 1$ we write  $\kC=C_0+C_1+\dots +C_{k+1}$, where $C_0, C_{k+1}$ are (-3)--curves,   $C_1,\dots, C_k$ are (-2)--curves  and $C_i$ intersects only $C_{i-1}$ and $C_{i+1}$
 transversely in one point. 

We start with a simple observation: 
\begin{lem}\label{lem:-4iota}
 Let $S$ be a Horikawa surface of the first kind of type (d) with $p_g=n+1\ge 4$ possessing a 2--Gorenstein T--chain $\kC$  and denote by $\fie_K$ the canonical map and by $\iota$ the canonical involution of $S$. Then there are four  possibilities:
\begin{itemize} 
\item[(L)] $\fie_K(\kC)$ is a ruling of $S_{d,n}$. In this case $\iota (\kC)=\kC$ but $\kC$ is not fixed by $\iota$ pointwise.
\item[(L$^\prime$)] $\fie_K(\kC)$ is a line that is not a ruling  of $S_{d,n}$. In this case $p_g\le 7$,  $\iota (\kC)=\kC$ but $\kC$ is not fixed by $\iota$ pointwise.
\item[(Q)] $\fie_K(\kC)$ is a smooth conic. In this case  $\kC$ is of type $[4]$ and  either $\iota (\kC)\ne \kC$ and $p_g\le 10$, or $\kC$ is  fixed by $\iota$ pointwise and $p_g\le 14$. 
\item[(Q$^\prime$)] $\fie_K(\kC)$ is a reducible conic. In this case  $p_g\le 10$, $\kC$ is not of type $[4]$ and  either $\iota (\kC)\ne \kC$ 
or $C_0$ and $C_{k+1}$ are  fixed by $\iota$ pointwise and $\iota(\kC)=\kC$.
\end{itemize}
\end{lem} 
\begin{proof}
Recall that $\fie_K$ is a morphism of degree $2$ that contracts precisely the $(-2)$-curves of $S$. If $\kC=C$ is of type $[4]$, we have $K_S\cdot C=2$ and therefore $C$ is either mapped $2$-to-$1$ to a line or it is mapped isomorphically to a conic. In the former case $\iota(C)=C$ but $C$ is not contained in the fixed locus of $\fie_K$ (so we are in case (L) or (L$^\prime$)),  while in the latter case either $\iota(C)\ne C$ or $C$ is contained in the fixed locus of $\iota$ (so we are in case (Q)).

If $\kC$ is of type $[3,2^k,3]$, then $K_S\cdot C_0=K_S\cdot C_{k+1}=1$ and therefore $C_0$ and $C_{k+1}$ are mapped to lines. Assume that $C_0$ and $C_{k+1}$ are mapped to the same line $R$ and therefore $C_0+C_{k+1}$ is $\iota$-invariant: the only non obvious claim is that in this case all the T--chain $\kC$ is $\iota$-invariant when   $k>0$.  Let $P\in R$  be the image point of $C_1+\dots+C_k$ and let $\Delta$ be the preimage of $P$, which  is an A-D-E configuration of $(-2)$-curves. Since $\fie_K$ maps $C_0$ and $C_{k+1}$ isomorphically to $R$, $C_0$ intersects $\Delta$ only at the point $C_0\cap C_1$ and  $C_{k+1}$  intersects $\Delta$ only at the point $C_k\cap C_{k+1}$. Since $\iota$ switches $C_0$ and $C_{k+1}$, it must  switch also $C_1$ and $C_k$.  Since the dual graph of $\Delta$ is a tree, there is only one chain of curves contained in $\Delta$ that has $C_1$ and $C_k$ as end points. So all the T-chain $\kC$ is preserved by $\iota$, and we are in case (L) or (L$^\prime$) again. 

If the lines $\fie_K(C_0)$ and $\fie_K(C_{k+1})$ are distinct, then we are in case (Q$^\prime$) and the only non obvious claim is that if $C_0$ and $C_{k+1}$ are contained in the fixed locus of $\iota$, then $\iota(\kC)=\kC$. This can be proven as above.  

Finally, note that in cases (L$^\prime$) and (Q$^\prime$) the canonical image contains  a line $\Gamma$ that is not a ruling of $S_{d,n}$. Such a line $\Gamma$  is a minimal section of $S_{d,n}$, so we have 
$$1= \deg \Gamma= \Gamma\cdot (D+\frac{n+d-1}{2}F)=D \cdot (D+\frac{n+d-1}{2}F)=\frac{n-d-1}{2},$$
  hence $n=d+3$ and, if $d>0$, $D=\Gamma$. In case (L$^\prime$) the image of the T--chain is not contained in the branch locus, so we  have $n+3\ge 3d$, namely $d\le 3$ and $p_g=n+1\le 7$. 

 In case (Q), $S_{d, n}$ contains a conic $\Gamma$ that is a section of the ruling. If $n\ge 5$, then $\Gamma=D$ is a minimal section. 
 If $\iota(C)\ne C$, then $\Gamma$ is not contained in the branch locus and then  $D\cdot (6D+(n+3+3d)F)=n+3- 3d\geq 0$. On the other hand, one has  
 $$2= \deg \Gamma= \Gamma\cdot (D+\frac{n+d-1}{2}F)= D \cdot (D+\frac{n+d-1}{2}F)=\frac{n-d-1}{2},$$
 and therefore $n=d+5$. Hence $d+5=n\ge3d-3$, namely $d\le 4$ and $p_g\le 10$. If $\Gamma$ is contained in the branch locus instead, one has $D\cdot (5D+(n+3+3d)F)=n+3-2d\geq 0$. Since $n=d+5$, we have  $d+8\geq 2d$ whence $d\leq 8$,  $n=d+5=13$ and so  $p_g\leq 14$.
 
 Finally  in case (Q$^\prime$), if  $\iota(\mathfrak C)\ne \mathfrak C$ exactly the same argument as above gives $p_g\leq 10$. If instead $\iota(\mathfrak C)= \mathfrak C$, then the line $\Gamma$ is contained in the  branch locus and so  $D\cdot (5D+(n+3+3d)F)=n+3-2d\geq 0$.  But in this case  we have  
 $n=d+3$  hence $d\leq 6$ and therefore  $n\leq 9$ and  $p_g\leq 10$.
 \end{proof}

\subsection{Case (L) of Lemma \ref{lem:-4iota}}
Since we are interested in Horikawa surfaces with unbounded invariants, in view of Lemma \ref{lem:-4iota} it suffices to analyze  in detail  only the case when the T-chain is mapped to a ruling of the canonical image of $S$.
Before stating our  result we recall some facts and set some notation. The canonical map of a Horikawa surface $S$ of the  first kind of type (d) induces a finite double cover $\bar S\to S_{d,n}\cong \F_d$, branched on a divisor $B$ with irrelevant singularities (Theorem \ref{thm:Ho1a}). 
So $S$  is  obtained from  this cover  by taking base change with  a suitable blow-up $\epsi\colon Y\to\F_d$ and then normalizing.  It follows that the quotient  $S/\iota$ by the canonical involution $\iota$ is smooth, isomorphic to $Y$. More precisely, $S$ is constructed as follows: one considers the blow-up $\epsi_1\colon Y_1\to\F_d$ of the singularities of $B$ and lets $S_1\to Y_1$ be the double cover branched on $B_1:=\epsi_1^*B-2\Delta_1$, where $\Delta_1$ is the total exceptional divisor of $\epsi_1$. If $B_1$ is smooth, then $S=S_1$, otherwise one repeats this procedure until one obtains a smooth branch divisor.

\begin{prop}\label{prop:-4} Let $S$ be a Horikawa surface of the  first kind  of type $(d)$  with $p_g=n+1\geq 4$ that contains  a 2-Gorenstein T-chain $\kC$ that 
is mapped by the canonical map to a ruling  $R$  of $S_{d,n}$. Then:
\begin{enumerate}
\item  if $\kC$ is of type $[4]$ or of type $[3,3]$, then  $R$ contains exactly two centers $P_1$, $P_2$ of the blow-up $\epsi\colon Y= S/\iota\to \F_d$;
\item if $\kC$ is of type $[3,2^k, 3]$ with $k\ge 1$, let $P\in R$ be the image of $C_1+\dots+C_k$; then $R$ contains exactly three centers $P_1$, $P_2$, $P_3$ of $\epsi$ where  $P_3$ is equal (or infinitely near) to  $P$, and $P_1,P_2$ are not infinitely near to $P_3$. 
\item   let  $\epsi'\colon Y'\to \F_d$ be the blow-up at $P_1,P_2$ and  let $E_1$, $E_2$ the corresponding  exceptional curves. Then $\kC$ is the pull back  on $S$ of $\Phi:={\epsi'}^*R-E_1-E_2$ and the following cases occur:
\begin{enumerate}
\item  $\kC=C$ is of type $[4]$ and  
 $\Phi$   intersects ${\epsi'}^*B-2E_1-2E_2$ at two distinct points
 \item $\kC=C$ is of type $[3,3]$   and $R'$  and  $\epsi^*B-2E_1-2E_2$ is  tangent at one point that is smooth for $\epsi^*B-2E_1-2E_2$
 \item $\kC$ is of type $[3,2^k, 3]$ with $k\ge 1$ and  
$\Phi$  meets  ${\epsi}'^*B-2E_1-2E_2$  at a double point of type $A_k$. 
\end{enumerate}
\end{enumerate}
\smallskip
\noindent Conversely, if a ruling $R$ of  the canonical image  of $S$ not contained in the branch locus of the canonical  map satisfies  condition (i)  [resp. (ii)] then the pull-back on $S$  of $\Phi:={\epsi'}^*R-E_1-E_2$  is a 2--Gorenstein  T-chain of type $[4]$ or $[3,3]$ [resp. $[3,2^k, 3]$ with $k\ge 1$].
\end{prop}

\begin{proof}  We do a case by case analysis according to the type of the T--chain. We denote by $R'$ the strict transform of $R$ via $\epsi$. 

  If $\kC=C$ is of type $[4]$ then the image of $C$ in $Y:=S/\iota$  is  $R'$. The projection formula gives  $(R')^2=-2$, namely $R$ contains exactly two of the centers of $\epsi$, say $P_1$ and $P_2$.
 Since the preimage of  $R'$ is the smooth rational  curve $C$, $R'$ meets the branch locus $B_Y$ of the double cover $S\to Y$ transversally at two distinct points or, equivalently, $\Phi$ and ${\epsi'}^*B-2E_1-2E_2$ meet transversally at two points. 
 \smallskip
 
 If $\kC$ is of type $[3,3]$ one shows as above that the image $R'$ of $\kC$ in $Y$ is a  $(-2)$-curve, so $R$ contains exactly two centers $P_1$ and $P_2$ of $\epsi$.
 Since the preimage of $R'$ is the union of two rational curves meeting at a point, $R'$ is tangent to the branch locus $B_Y$ of the double cover $S\to Y$ at one point or, equivalently, $\Phi$ and ${\epsi'}^*B-2E_1-2E_2$ are tangent   at one  point.
  \smallskip
 
 If $\kC$ is of type $[3,2^k, 3]$ with $k\ge 1$ the image $R'$ of $C_0+C_{k+1}$ in $Y$ is a (-3)-curve, so $R$ contains exactly three centers $P_1,P_2$ and $P_3$ of $\epsi$, and one of these coincides with  $P$.  Up to permuting the centers $P_1$, $P_2$, $P_3$ we may assume that neither $P_1$ nor $P_2$ is  infinitely near to $P_3$ and that $P_3$ is equal to (or infinitely near to) $P$.  The point $P_3\in \Phi$ is a singular  point of ${\epsi'}^*B-2E_1-2E_2$	 since it is a center of $\epsi$.  On the other hand $\Phi\cdot ({\epsi'}^*B-2E_1-2E_2)=2$ and $\Phi$ is not contained in  ${\epsi'}^*B-2E_1-2E_2$, so $P_3$ is double point of ${\epsi'}^*B-2E_1-2E_2$ that has no tangent in common with $\Phi$. If  $P_3$ is of type $A_m$, then the pull back of $\Phi$ to $S$ is a 2-Gorenstein T-chain   $\kC'$ of type $[3,2^m,3]$. So $\kC$ and $\kC'$ are T-chains with the same end curves and such that their chains of $(-2)$--curves are mapped to the same point of $\F_d$; arguing as in the proof of Lemma \ref{lem:-4iota}
one shows $m=k$ and $\kC=\kC'$.
\smallskip
 
The converse statement is easy and left to the reader. 
 \end{proof}
 \begin{rem}\label{rem:-4}
The simplest instance  of  case (iii.a) of Proposition \ref{prop:-4} is when $R$ is a ruling of $\F_d$ that  contains two ordinary double points $P_1$ and $P_2$ of $B$ and meets $B$ transversally at two additional points.

However, there are more complicated possibilities, such as  when $R$ contains an ordinary triple point $P_1$ of $B$,  is simply tangent  to one of the branches of $B$ at $P_1$ and meets $B$ transversally at two additional points. In this case the point $P_2$ infinitely near to $P_1$ in the direction of $R$ is a center of $\epsi$ lying on $R$ but it is not singular for (the strict transform of) $B$.

Similarly, the simplest instance of case (c) is when $R$ contains three double points of $B$, $P_1$ and $P_2$  of type $A_1$ and $P_3$  of type $A_k$.
\end{rem}

\subsection{Cases  (L$^\prime$), (Q), (Q$^\prime$)  of Lemma \ref{lem:-4iota}}

For completeness we discuss briefly the remaining cases of Lemma \ref{lem:-4iota}.  First we give some  examples to show that  all cases actually  occur and the bounds for $p_g$ are sharp. At the end of the section we  make some  remarks   on the KSBA moduli space of Horikawa  surfaces of the second kind with small $p_g$.
\medskip

In all the examples below we specify integers $d,n$ with $n-d$ odd and $n\ge\max\{ d+3, 2d-3\}$,  and a  divisor $B\in |6D+(n+3+3d)|$ on  $\F_d$ with negligible singularities.  We denote by $S$ the minimal resolution of the double cover $\bar S\to\F_d$ branched on $B$, which is a Horikawa surface of the first kind with $p_g=n+1$. Recall that $S$ is obtained by repeatedly blowing up $\F_d$, taking base change of $\bar S\to\F_d$ and normalizing.
\medskip

\noindent\underline{Case (L$^\prime$):}  If $d=2$, $n=5$ and $B$ is general,  then the negative section of $\F_2$  pulls back to a $(-4)$-curve on $S$, so the bound for $p_g$  given in Lemma \ref{lem:-4iota} is sharp. These examples 
give a family of dimension 55.

If $n=4$, $d=1$ and $B\in |6D+10F|$ has  a double point   $P\in D$ and is general otherwise, then the strict transform of $D$ on $S$ is a $(-4)$--curve. The construction depends on 47 moduli.
 \bigskip

\noindent\underline{Case (Q):} For  $0\le d\le 8	$ and $n=d+5$,  take $B=D+B_0$ with $B_0\in |5D+(4d+8)F|$.  The strict transform of $D$ on $S$ is a $(-4)$--curve, fixed pointwise by the canonical involution and mapped to a conic by the canonical map, so the bound for $p_g$  given in Lemma \ref{lem:-4iota} is sharp also in this case. For $d=8$ (and $n=13$) these surfaces   fill up the component $\mathcal M^{\rm Hor,1b}_{13}$ of the Gieseker moduli space, which has dimension 112 (cf. Theorem \ref{thm:Ho1b}). 

In case (Q) there is also the possibility that the $2$-Gorenstein T--chain is not fixed by the canonical involution.
   For $d=4$ and $p_g=n+1=10$ all the surfaces in $\mathcal M^{\rm Hor,1b}_{9}$ are examples of this situation: the preimage via the canonical map of the negative section of $\F_4$ is the disjoint union of two $(-4)$--curves. So the bound for $p_g$  given in Lemma \ref{lem:-4iota} is sharp also in this case.
 \bigskip
 
\noindent\underline{Case (Q$^\prime$):} 
We  give a couple of examples. For $d=5$ and $n=d+3$ we take a divisor $B=B_0+R$, where $R$ is a ruling and $B_0\in D+|5D+25F|$ is general.  The strict  transforms of $R$ and $D$ on $S$ are disjoint $(-3)$-curves connected by the $(-2)$--curve mapping to the intersection point of $R$ and $D$, so $S$ contains  a T--chain of type $[3,2,3]$ that is invariant under $\iota$.

For  $d=2$, $n=5$ and $B\in |6D+14 F|$ with an ordinary  double point $P$ on  $D$,  two additional double points on the ruling $R$ that contains $P$ and otherwise general,  $S$ contains  four disjoint $(-3)$--curves, two mapping to $R$ and two mapping to $D$. Each of these curves  intersects transversally at a point  the $(-2)$--curve mapping to $P$, so the corresponding Horikawa surface of the first kind  contains four  T--chains of type $[3,2,3]$ not invariant under $\iota$. 

\subsection{Smoothability and the canonical involution}

Let $S$ be  a Horikawa  surface  of the  first kind with a  2--Gorenstein T--chain $\kC$ and let $X$ be the stable Horikawa surface of the second kind obtained by contracting  $\kC$ and the $(-2)$--curves disjoint from it (cf. Remark \ref{rem: XS}).
We can  ask ourselves if such a surface  has a $\Q$-Gorenstein smoothing,  i.e., if it can be flatly deformed to a Horikawa surface of the second  kind in such a way that the total space of the deformation is $\Q$--Gorenstein (cf. \S \ref{ssec:KSBA}).

The canonical involution $\iota$ of $S$ descends to an involution of $X$  iff $\iota(\kC)=\kC$, and such an involution, when it exists,  acts trivially on the canonical system $|K_X|$.  Since   Horikawa surfaces of the second  kind also have a canonical involution, by \cite[Cor.~8.65]{kollar-families}  the existence of  the canonical involution is a necessary condition for the $\Q$--Gorenstein smoothability of $X$. So some of the examples above for case (Q) and for case (Q$^\prime$) give rise to non-smoothable stable Horikawa surfaces of the second kind.

The two   families of surfaces falling in case (L$^\prime$) that we described above give rise to families of stable Horikawa surfaces of the second kind whose only non-canonical singularity is a 2--Gorenstein T--singularity,  with $p_g=5$, respectively $p_g=6$, depending on 47, respectively 55, moduli.
On the other hand the number of moduli of  Horikawa surfaces of the second  kind with 
$p_g=5, K^2=7$ and $p_g=6, K^2=9$ are 47 and 54 respectively. From this we see that the above stable surfaces  fill up  irreducible components of the moduli space of stable surfaces and are generically not smoothable, though the contracted   $(-4)$--curve   is fixed by the canonical involution.

\section{Minimal surfaces with $K^2=2p_g-3$ with a 2-Gorenstein T--singularity}\label{sec:Hor-sing}

In this section we describe the stable Horikawa surfaces of the second kind  whose only non canonical singularity is a $2$--Gorenstein T-singularity ({\em $2$--Gorenstein $T$-singular  surfaces} for short).

We start by setting some notation. Given a stable    $2$--Gorenstein T--singular  Horikawa surface $X$ of the second kind with $p_g=n+1\ge 4$, we denote by  $f\colon S\to X$  the minimal desingularization and by $\kC$  the T--chain contracted by $f$. The surface $S$ is a minimal Horikawa surface of the first kind with $p_g=n+1$ (cf. Remark \ref{rem: XS}) and $f^*K_X=K_S+\frac 12 \kC$ (cf. \S \ref {ssec:Tsurf}).   We say that $X$ is of type (d) if $S$ is. One may also rephrase this definition by saying, as in the case of smooth surfaces, that the canonical image of $X$ is isomorphic to $\F_d$. In fact the pull back map $f^*\colon H^0(K_X)\to H^0(K_S)$ is an isomorphism, hence there is a commutative diagram

\[\xymatrix{
S\ar[r]^{f }  \ar[dr]_{\fie_{K_S}}&  X\ar@{.>}[d]^{\fie_{K_X}} \\
&  \pp^n}
\] 

For all $d$ such that $n-d$ is odd and $n\ge\max\{ d+3, 2d-3\}$, we define $\mathcal D_{n,d}\subset \mathcal M_{2n-1,n+2}^{\rm KSBA}$ to be the locus corresponding to the 2-Gorenstein T-singular surfaces of type (d) and  we set $$\mathcal D_n:=\sqcup_d\mathcal D_{n,d},$$ where  $d$ ranges over all the admissible values. 
 Recall (cf. \S \ref{sec:gen}) that, as soon as $n\ge 6$, all Horikawa surfaces of the first kind are of type (d), so $\mathcal D_n$ parametrizes  all the  2--Gorenstein T--singular stable Horikawa surfaces of the second kind with $p_g=n+1$.

In order to state the main result of this section, we need a little more notation. We denote by $\mathcal D_{n,d}^0\subset \mathcal D_{n,d}$, $\mathcal D_n^0\subset \mathcal D_n$ the subsets consisting of the surfaces with a $\frac 14(1,1)$ singularity.
Finally, for $d>0$ we decompose $\mathcal D_{n,d}=\mathcal D_{n,d}'\sqcup \mathcal D_{n,d}''$,  where a class $[X]\in \mathcal D_{n,d}$ belongs to $\mathcal D_{n,d}'$ [resp.  $\mathcal D_{n,d}''$] if the $f$-exceptional  T--chain  on $S$ meets [resp. does not meet] the  strict transform on  $S$ of the negative section $D$ of $\F_d$.
Assume that the T--chain falls  in case (L) of   Lemma \ref{lem:-4iota} and label $P_1, P_2$ in such a way that $P_1$ is a proper   point of $\F_d$ (i.e., not an infinitely near one) and $P_2\notin D$;  then $[X]$ belongs to $\mathcal D_{n,d}'$ if  $P_1\notin D$ and belongs to   $\mathcal D_{n,d}''$ otherwise (note that $\mathcal D_{2d-3, d}''=\emptyset$ since $D$ and $B-D$ are disjoint in this case). If $d=0$ we set  $\mathcal D_{n,0}'=\mathcal D_{n,0}$,  $\mathcal D_{n,0}''=\emptyset$. We denote by $\overline W$ the closure of a subset $W\subset \mathcal M_{2n-1,n+2}^{\rm KSBA}$.
\begin{thm}\label{thm:Dn}
 Notation as above.  Fix  $n\ge 14$   and let  $d$ be an integer such  that $n-d$ is odd and $n\ge\max\{ d+3, 2d-3\}$. 
 Then \begin{enumerate}
 \item $\mathcal D_{n,d}'$ and $\mathcal D_{n,d}''$ are constructible
 \item  $\mathcal D_{n,d}^0$ is dense in $\mathcal D_{n,d}$. 
\item  $\mathcal D_{n,d}'$  and $\mathcal D_{n,d}''$ are irreducible. Their dimensions are listed   in Table \ref{tab:Dn}\footnote{A check mark in the  third [resp. fifth] column of Table \ref {tab:Dn}  indicates that ${\mathcal D_{n,d}'}$ [resp. ${\mathcal D_{n,d}''}$] is a component of $\mathcal D_{n,d}$. The symbol \xmark\, in the fifth column indicates that ${\mathcal D_{n,d}''}$ is not a component of $\mathcal D_{n,d}$. }, together with the indication of whether or not they are  irreducible components of ${\mathcal D_{n,d}}$.
\end{enumerate}
 \begin{table}[h!]
 \caption{}
 \label{tab:Dn}
  \begin{center}
    {\rm
    \begin{tabular}{|c|c|c|c|c|c|}       \hline
       & $\dim \mathcal D_{n,d}'$  & irr. comp.?       &  $\dim \mathcal D_{n,d}''$ &  irr. comp.?\\
       \hline
      $d=0$    & 7n+18 &\cmark   & -1 &\xmark\\
       \hline

       $n\ge 3d-1, d\ge 1$  & $7n+19-d$&  \cmark & $7n+18-d$  &\xmark\\
      \hline
     $ n=3d-3$   & $7n+19-d$& \cmark & $7n+19-d$  &\cmark\\
      \hline
       $ 2d-2\le n< 3d-3$  & $6n+2d+15$ & \cmark  & $6n+2d+16$  & \cmark\\
      \hline
          $ 2d=n+3$ & $7n+18$  & \cmark & -1 & \xmark  \\
      \hline

    \end{tabular}}
  \end{center}
\end{table}

 In particular, the 2-Gorenstein T-singular surfaces with a singularity of type $[3,2^k,3]$, $k\ge 0$,  appear in dimension at most $7n+17$ in $\mathcal M_{2n-1,n+2}^{\rm KSBA}$.
 \end{thm}
 
\begin{rem} \label{rem:top-dimension-strata} Recall that $\mathcal M^{\rm Hor,2}_n$ has dimension $7n+19$  (cf. Theorem \ref{thm:Ho23b}). By Table \ref{tab:Dn}, $\dim\mathcal D_{n,d}\le 7n+18$ for all $d$  and there is  always at least  one stratum  $\mathcal D_{n,d}$  of dimension $ 7n+18$, namely either $\mathcal D_{n,0}$ or $\mathcal D_{n,1}$ according to whether $n$ is odd or even. Depending on the class of $n$ modulo 4 there is an additional codimension 1 stratum:   $\mathcal D_{4k+1, 2k+2}=\mathcal D_{4k+1, 2k+2}'$   for $n=4k+1$ and $\mathcal D_{4k, 2k+1}''$  
for  $n=4k$ (cf. line 5, respectively line 4 for $n=2d-2$, of Table \ref{tab:Dn}). 
 \end{rem}

\subsection{Proof of Theorem \ref{thm:Dn}}\label{ssec:proof1}

    If   $X$ is a surface in $\mathcal D_n$ with  minimal desingularization $ \eta\colon S\to X$  and $\mathfrak C$ is the $\eta$--exceptional divisor, then   the pair $(S,\mathfrak C)$ arises  as in case (L) of Lemma \ref{lem:-4iota} since $n\ge 14$ by assumption.

We use freely the notation introduced in  Appendix \ref{sec:appendix1}. Fix a ruling $R$ of $\F_d$  and for $i,j\in\{0,1\}$  consider the system $|L_{ij}|$ on the blow-up $Y_{ij}\to \F_d$; assume furthermore that the $4$-tuple   $(n,d,i,j)$ does not fall in case (i) of Proposition \ref{prop:P1P2}, namely that the general curve of $|L_{ij}|$ is reduced. By Proposition \ref{prop:P1P2} the general curve of $|L_{ij}|$ does not contain the strict  transform $R'$ of $R$ and has at most nodes as singularities. So the open set $U_{ij}\subset |L_{ij}|$ consisting of divisors with at most negligible singularities and not containing $R'$  is nonempty.
Given any $\bar B\in U_{ij}$, let $\bar S\to Y_{ij}$ be the double cover branched on $\bar B$ and let $S$ be the minimal resolution of $\bar S$. The surface $S$ is a Horikawa surface of the first kind and the pull-back $\kC$ of $R'$ is a 2--Gorenstein T--chain on $S$. The surface $X$ obtained from $S$ by contracting $\kC$ and all the $(-2)$--curves disjoint from $\kC$ is in $\mathcal D_{n,d}$ and by Proposition \ref{prop:-4} and Proposition \ref{prop:XS} every surface in $\mathcal D_{n,d}$ arises in this way for a suitable choice $i,j\in\{0,1\}$.

 Denote by $\pi\colon \pp^{r_{ij}}\to |L_{ij}|$ the $\Z_2^{r_{ij}}$-cover branched on $r_{ij}+1$ hyperplanes in general position, where $r_{ij}:=\dim |L_{ij}|$.  Choosing suitable homogeneous coordinates on $\pp^{r_{ij}}$ and $|L_{ij}|$ the morphism $\pi$ can be written as $[z_0,\dots z_{r_{ij}}]\mapsto [z_0^2,\dots z_{r_{ij}}^2]$, so $\pi^*M=M^{\otimes 2}$ for all $M\in \Pic(|L|)$.

\begin{lem} \label{lem:stable} Set $V_{ij}:=\pi\inv(U_{ij})$.
There is a   stable family $g_{ij}\colon \Xca_{ij} \to V_{ij}$ such that for all $v\in V_{ij}$ the fiber $g_{ij}\inv(v)$ is the T--singular  stable Horikawa surface of the second kind corresponding to  $\pi(v)$.
 \end{lem}
 \begin{proof}  In the proof we drop the indices ${i,j}$ from the notation for simplicity.
 
Let $\Bca$ be the universal divisor on $Y\times |L|$ and let $\pr_1\colon Y\times |L|\to Y$ and  $\pr_2\colon Y\times |L|\to |L|$ be the projections. The divisor $\Bca-2({\pr}_1^*(\epsi^*(3D+\frac{n+3+3d}{2}F)-E_1-E_2)$ is linearly equivalent to zero on the fibers of $\pr_2$,  so there is a line bundle $M$ on $|L|$ such that $\Bca-2({\pr}_1^*(\epsi^*(3D+\frac{n+3+3d}{2}F)-E_1-E_2)=\pr_2^*M$.  Taking  base change with  $\pi$   we obtain a  family $Y\times \pp^r\to \pp^r$ such that  $\Bca$ pulls back to a divisor $\widetilde{\Bca}$ that is divisible by 2 in $\Pic(Y\times \pp^r)$. 

We then take a flat double cover of $Y\times \pp^r$ branched on $\widetilde{\Bca}$ and restrict it to $Y\times V$.  We denote by $f\colon \Sca\to V$ the family thus obtained: $\Sca$ is Gorenstein, $f$ is flat and the
 fibers  $S_{v}$ are surfaces  with at most canonical singularities such that $K_{S_v}$ is  nef and $K_{S_{v}}^2=2n-2$. In other words, the minimal resolution of  $S_v$  is a Horikawa surface of the first kind. 
  Denote by   $\mathcal C$  the pull back on $\Sca$ of $R'\times V\subset Y\times V$. The restriction of $\mathcal C$ to a fiber $S_v$ of $f$ pulls back to a 2--Gorenstein T--chain $\kC_v$ on the minimal resolution of $S_v$ (cf. Proposition \ref{prop:-4}).

 Now let $\mathcal L:=2K_{\Sca/V}+\mathcal C$ and set $\Xca:=\Proj_V(\oplus_{m\ge 0}f_*(m\mathcal L))$: by Lemma \ref{lem:can-ring},  $f_*(m\mathcal L)$ is locally free for all $m\ge 0$ and the fiber over $v\in V$ of the induced map $g\colon \Xca\to V$ is the stable Horikawa  surface of the second kind $X_v$ obtained from the minimal resolution of $S_v$ by contracting    the T--chain $\kC_v$ and the $(-2)$-curves disjoint from it.  The tautological sheaf $\OO_{\Xca}(1)$ restricts to $2K_{X_v}$ on $X_v$ for all $v\in V$ and so  $g_*(\OO_X(m))$ is locally free, therefore \cite[Thm.~3.20]{kollar-families} implies that $g$ is flat. 
 Finally $K_{X_v}^2=2n-1$ is constant, independent of $v$, so $g$ is stable by \cite[Thm.~5.1]{kollar-families}.
 \end{proof}
  By Lemma \ref{lem:stable} there is a morphism $V_{ij}\to \mathcal M^{\rm KSBA}_{n+2,2n-1}$ that descends to a morphism $\phi_{ij}\colon U_{ij}\to \mathcal M^{\rm KSBA}_{n+2,2n-1}$. We have $\mathcal D'_{n,d}=\phi_{00}(U_{00})\sqcup \phi_{01}(U_{01})$ and $\mathcal D''_{n,d}=\phi_{10}(U_{10})\sqcup \phi_{11}(U_{11})$, 
     with the exception of   the cases $n=2d-3$ and $d=0$, when $\mathcal D''_{n,d}=\emptyset$, of  $n=2d-2$, when $\mathcal D''_{n,d}=\phi_{10}(U_{10})$. This proves (i).
  
    Whenever the general curve of $|L_{ij}|$ is reduced, the curves in $U_{ij}$  that meet $R'$ transversally at two points form a nonempty open set $U_{ij}^0$ by Proposition \ref{prop:P1P2}. The image of $U_{ij}^0$ via $\phi_{ij}$ is contained in $\mathcal D_{n,d}^0$, so $\mathcal D_{n,d}^0$ is dense in $\mathcal D_{n,d}$ and also claim (ii) is  proven. 
  
\begin{lem}\label{lem:01} 
 \begin{enumerate}
 \item For $i=0,1$ and $d>0$,  the set $\phi_{i1}(U_{i1})$  is contained  in the closure of $\phi_{i0}(U_{i0})$.
 \item For $n\ge 3d-1$, the set $\phi_{10}(U_{10})$  is contained  in the closure of $\phi_{00}(U_{00})$.
 \end{enumerate}
   \end{lem}
  \begin{proof} (i) It is enough to show that a  surface $X$  in $\phi_{i1}(U_{i1})$ is a limit of surfaces in $\phi_{i0}(U_{i0})$.   It is not hard to construct explicitly a smooth family $h\colon \Yca\to T$ of surfaces  such that:
  \begin{itemize}
  \item  $T$ is  a smooth curve and
  \item  there is $\bar t\in T$ such that $Y_{\bar t}=Y_{i1}$ and $Y_{t}$ is isomorphic to $Y_{i0}$ if $t\ne \bar t$.
    \item $\Yca$ carries a line bundle $\mathcal L$ whose restriction to $Y_t$ is isomorphic to $L_{i0}$ if $t\ne \bar t$ and is isomorphic to $L_{i1}$ if $t=\bar t$.
    \end{itemize}
  By Proposition \ref{prop:P1P2} $h^0(\mathcal L|_{Y_t})$ is constant, hence $h_*{\mathcal L}$ is a vector bundle. Therefore, up to shrinking $T$, the divisor $\bar B\in U_{i1}$ on $Y_{i1}$ that maps to $[X]\in \mathcal D_{n,d}$ can be extended to a divisor $\mathcal B$ on $\Yca$ such that the restriction of $\mathcal B$ to $Y_t$ is in $U_{i0}$ for $t\in  T$, $t\ne \bar t$. Arguing as in the proof of Lemma \ref{lem:stable}, one shows that, up to further shrinking $T$ and taking base change with a finite cover, there is a  stable family $\mathcal X\to T$ of  Horikawa surfaces of the second kind  such that $X_{\bar t}=X$ and $[X_t]\in \phi_{i0}(U_{i0})$ for $t\ne \bar t$.  
  \medskip
  
 (ii)  The proof is very similar to that of (i) and so we leave it to the reader. The key point is that for $n\ge 3d-1$ the systems $|L_{00}|$ and $|L_{10}|$ have the same dimension. \end{proof}
     
Lemma \ref{lem:01} implies that $\mathcal D_{n,d}'$ and $\mathcal D_{n,d}''$ are irreducible, 
 of dimension  equal to the dimension of $\phi_{00}(U_{00})$, respectively $\phi_{10}(U_{10})$. To compute these dimensions we note first of all that $\bar B_1, \bar B_2\in U_{i0}^0$ give rise to isomorphic surfaces $\bar X_1$, $\bar X_2$ if and only if there is an automorphism of $Y_{i0}$ mapping  $\bar B_1$   to $\bar B_2$. In fact,    denoting  by  $\bar S_j\to Y_{i0}$ the double cover branched on $\bar B_j$, $j=1,2$, the surface $\bar S_j$ is the minimal resolution of the (only) non--canonical singularity of $X_j$ and the covering map is induced by the canonical involution. 
 Hence $\dim \mathcal D'_{n,d}=\dim |L_{00}|-\dim\Aut(Y_{00})$ and $\dim \mathcal D''_{n,d}=\dim |L_{10}|-\dim\Aut(Y_{10})$. So the dimensions can be computed using Table \ref{tab:M} and recalling that:
 \begin{itemize}
 \item  $\dim\Aut(Y_{00})=3$ if $d=0$ and $\dim\Aut(Y_{00})=d+2$ otherwise
 \item  $\dim\Aut(Y_{10})=d+3$ for $d>0$.
 \end{itemize}

The last step in the proof consists in determining whether $\ol{\mathcal D_{n,d}'}$ and $\ol{\mathcal D_{n,d}}''$ are irreducible components of $\ol{\mathcal D_{n,d}}$. 

  Lemma \ref{lem:01} implies that   if  $n\ge 3d-1$ then  $\ol{\mathcal D_{n,d}}$ is equal to $\ol{\mathcal D'_{n,d}}$  and therefore is irreducible.

 For $n=3d-3$ the sets $\mathcal D_{n,d}'$ and $\mathcal D_{n,d}''$ have the same dimension and therefore the closures of both are components of $\ol{\mathcal D_{n,d}}$. 
 
If  $ 2d-2\le n< 3d-3$ we have  $\dim\mathcal D_{n,d}''> \dim \mathcal D_{n,d}'$, so the question is whether $\mathcal D_{n,d}'$ lies in the closure of $\mathcal D_{n,d}''$. 
We are going to show that the  answer is no. In order to do this we need the following observation:
 \begin{lem} Let $X$ be a  surface in $\phi_{ij}(U_{ij}^0)$ and let $\iota$ be the canonical involution of $X$. The surface $X/\iota$
is isomorphic to the surface $\ol{Y_{ij}}$ obtained from $Y_{ij}$ by contracting  $R'$ to a singular point of type $A_1$
\end{lem}
\begin{proof}  The surface $X$ is obtained from a double cover  $\bar S\to Y_{ij}$ by contracting the preimage $C$ of $R'$, which is a $(-4)$-curve of $S$. The composition $\bar S\to Y_{ij}\to \ol{Y_{ij}}$ contracts $C$ and therefore induces a degree 2 morphism $\pi \colon X\to \ol{Y_{ij}}$. The morphism $\pi$ is $\iota$-invariant, so it induces a birational morphism $X/\iota\to \ol{Y_{ij}}$. Since both surfaces are normal, this is actually an isomorphism. 
\end{proof}
Assume now by contradiction that  given a general surface $X$ in  $\mathcal D_{n,d}'$ there is a
 stable  family $\Xca\to T$, with $T$ a smooth curve, such that $X_{\bar t}=X$ for some $\bar t\in T$ and $X_t$ is general in  $\mathcal D_{n,d}''$ for all $t\in T$ with $\bar t\ne t$. Taking the quotient by the canonical involution we obtain a family $\Yca\to T$ such that $Y_{\bar t}=\ol{Y_{00}}$ and $Y_t$ is isomorphic to $\ol{ Y_{10}}$ for $t\ne \bar t$. By \cite[\S 2.1]{kollar-families}, there is  a relative canonical sheaf $K_{\mathcal Y/T}$ which restricts to $K_{Y_t}$  on each fiber $Y_t$. In particular   $K_{\mathcal Y/T}$ is invertible (hence flat over $T$) because all  fibers $Y_t$ are Gorenstein. 
  On the other hand  $d\geq 6$, because we are assuming $n\ge 14$  and $n< 3d-3$,  hence standard 
computations, similar to those in Appendix \ref{sec:appendix1}, give $h^0(-K_{{\ol Y}_{00}})=h^0(-K_{Y_{00}})=d+4$ and  $h^0(-K_{\ol{Y}_{10}})=h^0(-K_{Y_{10}})=d+5$,  contradicting to the Semicontinuity Theorem. So $\ol{\mathcal D_{n,d}}$ has two components in this case. 

Finally, for $n=2d-3$ the subset $\mathcal D_{n,d}''$ is empty by definition, so there is nothing to prove and the proof is complete.

\begin{rem}\label{rem:Dn-dim}
Let $X$ correspond to a general point in an irreducible component of $\mathcal D_{n,d}$: then $X$ can have some singularities of type $A_1$ besides the $1/4(1,1)$ singularity.  Denoting by $\nu$ the number of these singularities Theorem \ref{thm:Dn} and its proof show:
\begin{itemize}
\item if  $n\ge 3d-3$ and for $n=2d-3$ one has  $\nu=0$ 
\item if $2d-2\le n\le 3d-5$ then $\nu = n+3-2d$ if $[X]$ belongs to   $\mathcal D_{n,d}'$ and $\nu = n+2-2d$  if $[X]$ belongs to   $\mathcal D_{n,d}''$. 

\end{itemize}
So, looking at Table \ref{tab:Dn}, we see that if $\nu>0$  $[X]$ belongs to a family of  equisingular deformations depending on $7n+18-\nu$ moduli. 
\end{rem}

\section{Explicit degenerations and smoothings of Horikawa surfaces}\label{sec:explicit}

Here we are going to describe explicitly degenerations of Horikawa surfaces of the second kind with canonical singularities to stable surfaces with a $\frac14(1,1)$-singularities and $\Q$--Gorenstein  smoothings of the general surface in (one of the components of) $\mathcal D_{n,d}$. It is worth noting (cf. Proposition \ref{prop:explicit}) that  this  method  works only  for some of  the components of the strata of $\mathcal D_{n,d}$ described in Section \ref{sec:Hor-sing}.  In particular it does not work in the case $n=2d-3$, that gives  a top dimensional stratum of $\mathcal D_n$.   In the next section we are going to prove by deformation theoretic methods that all the surfaces in $\mathcal D_n$ admit a $\Q$--Gorenstein smoothing.

We use freely the notation introduced in the previous sections and in Appendix \ref{sec:appendix1}.

 As we saw in Remark \ref {rem:liup}, for fixed $n,d$ such that $n-d$ is odd and $n\ge\max\{d+3, 2d-2\}$  the general surface in 
         $\mathcal M_{n,d}^{\rm Hor,  2}$ is a double cover $g\colon \bar S\to\ol{ Y}_{i0}$ where:
 \begin{itemize}
 \item  $i=0$  if $n\ge 3d-3$ and $i=1$ if $2d-2\le n\le 3d-5$. (However in the degeneration we will also consider surfaces with $i=1$ and $n=3d-3$).
 \item $\ol{ Y_{i0}}$ is obtained from $Y_{i0}$ by contracting $R'$ to  a point $x$ of type $A_1$. 
 \item $g$ is branched at $x$ and along a curve $\mathfrak B$ with irrelevant singularities not containing $x$. More precisely, $\mathfrak B$ is  the image of a general element of $|N_{i0}|:=|\epsi_{i0}^*(6D+(n+4+3d)F)-3E_1-3E_2|$.  
 \end{itemize} 
Note that $|N_{i0}|$ contains the subsystem $R'+|L_{i0}|$. Choose $B_0\in |L_{i0}|$ general and choose a general pencil containing  $B_0+R'$ in $|N_{i0}|$. Using such a pencil and arguing as in the proof of Lemma \ref{lem:stable}   one can construct a  double cover  $\mathcal X_{0} \to Y_{i0}\times T$, with $T$ a smooth curve, branched on a  divisor $\mathcal B$ that restricts to $B_0+2R'$ on $Y_{i0}\times\{\bar t\}$ for some  $\bar t\in T$ and restricts to $B_t+R'$ with $B_t\in |N_{i0}|$ general for all $t\ne \bar t$; in particular, for $t\ne \bar t$ the divisor $B_t$ has negligible singularities and does not contain $R'$. 

We can restrict $\mathcal X_0$ to $(Y_{i0}\setminus R')\times T=(\ol{Y}_{i0}\setminus\{x\})$ and then take its $S_2$-closure  $\mathcal X\to \ol{Y}_{i0}\times T$. The resulting family $\mathcal X\to T$ is flat and $\mathcal X$ is Cohen-Macaulay by Lemma \ref{lem:CM}. So the fibers $X_t$ of $\mathcal X\to T$ are $S_2$ and therefore they are  double covers of $\ol{Y}_{i0}$ branched on $B_t$ and the singular point $x$ for $t\ne \bar t$ and they are Horikawa surfaces of the second kind of type (d) by Remark \ref{rem:liup}.

 Consider now the fiber $X_{\bar t}$ over $\bar t$. It is the double cover of $\ol{Y}_{i0}$ branched on the image $\bar B_0$ of $B_0$ in $\ol{Y}_{i0}$. There is a commutative diagram, obtained by base change plus normalization:
 \begin{equation}\label{diag:explicit}
\xymatrix{
\bar S\ar[r]^{f }  \ar[d]_{\pi}&  X_{\bar t}\ar[d]^{\bar \pi} \\
Y_{i0} \ar[r]_{\eta}&  \ol{Y}_{i0}}
\end{equation}
where $\pi$ is a double cover branched on $B_0$, $\bar\pi$ is a double cover branched on the image $\bar B_0$ of $B_0$ and $\eta$ is the contraction of $R'$ to an $A_1$ point. 
The surface $\bar S$ is the canonical model of a Horikawa surface of the first kind of type (d) and the preimage $C$ of $R'$ is a $(-4)$--curve. The morphism $f$ contracts the preimage $C$ of $R'$ and is an isomorphism on the complement of $C$. Since $X_{\bar t}$ is normal, this shows that $X_{\bar t}$ is a stable 2--Gorenstein T--singular  Horikawa surface 
of the second kind of type (d) (cf. Proposition \ref{prop:XS}).
More precisely, $X_{\bar t}$ corresponds to a general point of $\mathcal D'_{n,d}$ if $n>3d-4$ and to a general point of $\mathcal D_{n,d}''$ if $2d-2\le n<3d-4$.
So all the fibers of $\mathcal X\to T$ are stable surfaces and $K_{X_t}^2$ is constant, so $\mathcal X\to T$ is a stable family by  \cite[Thm.~5.1]{kollar-families}.

Since $B_0$ is chosen general, one can arrange that $X_{\bar t}$ is any surface in $\phi_{00}(U_{00})$ if $n>3d-4$ and corresponds to  a general point of $\phi_{10}(U_{10})$ if $2d-2\le n<3d-4$ (cf. \S \ref{ssec:proof1}). The above discussion is summarized and completed in the following:

\begin{prop}\label{prop:explicit}
 Let $n,d\in \mathbb N$ be such that $n-d$ is odd and $n\ge \max\{d+3, 2d-2\}$; denote by $\ol{\mathcal  M_{n,d}^{\rm Hor,2}}$ the closure of $\mathcal  M_{n,d}^{\rm Hor,2}$. Then 
\begin{enumerate}
\item if $n\ge 3d-3$, then $\ol{\mathcal  M_{n,d}^{\rm Hor,2}}$ contains $\mathcal D_{n,d}$;
\item if $2d-2\le n\le 3d-5$, then $\ol{\mathcal  M_{n,d}^{\rm Hor,2}}$ contains $\mathcal D_{n,d}''$ and does not contain $\mathcal D_{n,d}'$
\end{enumerate} In particular $\ol{\mathcal D_n}\cap \ol{\mathcal  M_{n,d}^{\rm Hor,2}}$ has at least a codimension 1 component.
\end{prop}
\begin{proof} The first assertion of part (ii) and part  (i) have already been proven. To complete the proof  of (ii), we just have to show that for  $n\le 3d-5$ the general surface in $\mathcal D_{n,d}'$ cannot be a stable limit of surfaces in $\mathcal  M_{n,d}^{\rm Hor,2}$. This can be done by using the argument at the end of Section \ref{ssec:proof1}.

The final claim follows by Theorem \ref{thm:Ho23b} and Theorem  \ref{thm:Dn}.
\end{proof}

\section{Local structure of $\mathcal M_{2n-1,n+2}^{\rm KSBA }$   near a general point of $\mathcal D_{n,d}$}\label{sec:infinitesimal}

We keep the notation from the previous sections.

\subsection{Statement of the results}\label{ssec:results} 
The main results of this section are the following:
\begin{thm} \label{thm:smooth}
Let $n$ and $d$ be nonnegative integers such that $n-d$ is odd and  $n\ge\max\{2d-3, d+3\}$.
Let $X$ be general in an irreducible  component of $\mathcal D_{n,d}$.
Then:
\begin{enumerate} 
\item  $[X]$ belongs to  the closure $\ol{\mathcal M_{2n-1,n+2}}$  of $\mathcal M_{2n-1,n+2}$  inside $\mathcal M_{2n-1,n+2}^{\rm KSBA}$ 
and $\mathcal M_{2n-1,n+2}^{\rm KSBA}$ is smooth at $[X]$ 
\item  $\ol{\mathcal D_n}$  is a  smooth divisor at $[X]$.
 \end{enumerate}
\end{thm}
Combining Theorem \ref{thm:smooth} with the results in the previous sections we obtain:
\begin{thm} \label{thm:smoothD}
 Let $n\ge 14$ be an integer and let $\mathcal D_n\subset {\mathcal M^{\rm KSBA}_{2n-1,n+2}}$ be the locus of 2-Gorenstein  T--singular stable surfaces. Then
\begin{enumerate}
\item $\ol{\mathcal D}_n$ is a generically Cartier reduced  divisor contained  in $\ol{\mathcal M_{2n-1,n+2}}$.
\item $\ol{\mathcal D}_n$ is irreducible  if $n\equiv 2,3$ modulo $4$ and it has two irreducible components otherwise.   When $\ol{\mathcal M_{2n-1,n+2}}$ has two irreducible components (i.e., when $n\equiv 0$ modulo $4$), then $\ol{\mathcal D}_n$ intersects each of them  in an irreducible divisor.
\item If $X$ is general in a component of $\ol{\mathcal D}_n$, then the only non canonical singularity of $X$ is of type $\frac 14(1,1)$.
\end{enumerate}
\end{thm}
\begin{proof}
Recall that for $n\ge 7 $ we have  $\mathcal M_n^{\rm Hor, 2}=\mathcal M_{2n-1,n+2}$ (cf. Theorem \ref{thm:Ho23b}) and that for $n\ge 14$ the locus $\mathcal D_n$ coincides with $\mathcal D_n$.

 (i) Follows directly from Theorem \ref{thm:smooth} since $\mathcal D_n=\sqcup_d \mathcal D_{n,d}$, where $d-n$ is odd and $n\ge \{d+3, 2d-3\}$.

(ii)  The  irreducible components of $\ol{\mathcal D_n}$ correspond to the values of $d$ such that $\ol{\mathcal D_{n,d}}$   has a component of dimension $7n+18$, and have been determined in Remark \ref{rem:top-dimension-strata}. By Proposition \ref{prop:explicit} if $n=4k$, then $\ol{\mathcal M_n^{\rm Hor, 2a}}$ contains $\mathcal D_{n,1}$ and $\ol{\mathcal M_n^{\rm Hor, 2b}}$
contains $\mathcal D_{n,2k+1}$.

(iii) Follows by Theorem \ref{thm:Dn}.
\end{proof}

The proof of Theorem \ref{thm:smooth} is based on the following crucial technical result:
\begin{thm} \label{thm:TX}
In the assumptions of Theorem \ref{thm:smooth}, consider the natural action  of the canonical involution $\iota$ on the tangent sheaf $T_X=\Hom(\Omega^1_X,\OO_X)$. Let $\nu$ be the number of $A_1$-points of $X$. Then:
\begin{enumerate}
\item   $H^1(T_X)$ is $\iota$-invariant and has dimension
  $h^1(T_X)=7n+18-\nu$.
  \item $H^2(T_X)$ is anti-invariant and   has dimension
  $h^2(T_X)=n-3$.
\end{enumerate}
\end{thm}

We prove now Theorem \ref{thm:smooth}, granting Theorem \ref{thm:TX}, whose proof occupies the next section.

\begin{proof}[Proof of Theorem \ref{thm:smooth}]

If $V$ is a vector space on which $\iota$ acts we denote by $V^+$ the invariant subspace and by $V^-$ the anti-invariant one.

The involution $\iota$ acts naturally on the local-to-global sequence
\ref{eq: local-to-global} associated to the functor of $\Q$--Gorenstein  deformations $\Def_{QG}(X)$. This action,  taking into account Theorem \ref{thm:TX}, gives two exact sequences:
\begin{equation} \label{eq:+}
0\to H^1(T_X)^+\to(\T_{QG}^1)^+\to H^0(\Tc_{QG}^1)^+\to H^2(T_X)^+(= 0)\to  (\T_{QG}^2)^+ \to 0,
\end{equation}
and
\begin{equation} \label{eq:-}
0\to H^1(T_X)^-(=0)\to (\T_{QG}^1)^-\to H^0(\Tc_{QG}^1)^-\to H^2(T_X)^-\to (\T^2_{QG})\to 0,
\end{equation}
Sequence \eqref{eq:+} shows  that $(\T^2_{QG})^+=0$, so the functor $\Def_{QG}(X, \iota)$ of $\Q$--Gorenstein  deformations preserving the action of $\iota$ is smooth. 
 Next we look at $\Tc^1_{QG}$: as explained in Section \ref{ssec: Tdef}, it is a skyscraper sheaf concentrated at the singular points of $X$ and has length 1 at each point, because $X$ is general in $\mathcal D_{n,d}$ by assumption. Consider the double cover  $\pi\colon X\to \bar Y:=X/\iota$ (cf. diagram \eqref{diag:explicit}). The morphism $\pi$ represents the $\frac 14(1,1)$--singularity as a double cover of an $A_1$--singularity $w^2-uv=0$ branched on the intersection with $w=0$: since an $S_2$ double cover with Cartier branch locus is determined \'etale locally by the branch divisor (\cite[Lem.~1.2, Prop.~1.6]{non-normal}) the action of $\iota$ is \'etale  locally isomorphic to that of the involution $\bar \tau$ described in Remark \ref{rem:sing-inv}. Similarly, the $A_1$ singularities of $X$ map to smooth points of $\bar Y$ that are ordinary double points of the branch locus, so in local coordinates $X$ is given by $z^2-xy=0$, $\pi$ is given by $(x,y,z)\mapsto (x,y)$ and $\iota$ acts by $(x,y,z)\mapsto(x,y,-z)$. So in both cases  the local  situation is the one considered in Remark \ref{rem:sing-inv} and 
   $H^0(\Tc_{QG}^1)$ is $\iota$-invariant. Therefore $(\T_{QG}^1)^-=0$ by sequence \eqref{eq:-}. So  $0=(\T^2_{QG})^+\to \T^2_{QG}$ is injective,  $(\T^1_{QG})^+\to (\T_{QG}^1)$ is an isomorphism, and therefore the map $\Def_{QG}(X,\iota)\to \Def_{QG}(X)$ is smooth (and thus  \'etale) by a well known criterion (cf. \cite[Prop.~2.3.6]{sernesi}) and $\Def_{QG}(X)$ is also smooth.  
 
The map   $\T_{QG}^1\to H^0(\Tc^1_{QG})$ is surjective, so $X$ has a $\Q$--Gorenstein smoothing,  the tangent space to $\mathcal D_n$ at $[X]$ is the kernel of the nonzero map $\T^1_{QG}\to \Tc^1_{QG,x}$, where $x\in X$ is the only non-canonical singularity.  So, by (i) of Theorem \ref{thm:TX}, this kernel has dimension $7n+18$   and $\mathcal D_n$ is a smooth Cartier divisor at the point $[X]$. 
\end{proof}

\begin{rem} The assumption ``$X$ general in $\mathcal D_{n,d}$'' in the statement of Theorem \ref{thm:TX} (hence in all the results of  section \ref{ssec:results}) can be made explicit, as follows. The divisor $\bar B\in |6D+(n+3+3d)F|$ corresponding to $X$ must satisfy the following conditions:
\begin{itemize}
\item the intersection of $\bar B$ with a fixed ruling $R$ of $\F_d$ consists of four distinct points: two of these, called $P_1$, $P_2$ in the proof are ordinary double points of $\bar B$, while the remaining two are smooth points of $\bar B$; 
\item if $n> 3d-3$ or $n=3d-3$ and  $P_1,P_2\notin D$,  or $n=2d-3$, then the points $P_1$, $P_2$ are the only singularities of $B$;
\item if $2d-3<n<3d-3$ or $n=3d-3$ and  $P_1$ or $P_2$ lies on $D$, then $\bar B=D+\bar B_0$ where $\bar B_0$ is irreducible and the only singularities of $B$ besides $P_1,P_2$ are ordinary double points occurring  at the intersection of $D$ and $\bar B_0$.
\end{itemize}
\end{rem}

\subsection{Proof of Theorem \ref{thm:TX}}

We keep the notation of the previous sections.

As we have seen in Section \ref{sec:Hor-sing} the surface $X$, being general, is obtained as follows. 
We let $R$ be a ruling of $\F_d$,   $P_1, P_2\in R$  distinct points with $P_2\notin D$  and we let $\bar B\in|6D+(n+3+3d)F)|$ be a divisor general among those with a double point at $P_1$ and $P_2$ (i.e., in the notation of Appendix \ref{sec:appendix1} the divisor $\bar B$ corresponds to a general element of $|L_{i0}|$ for $i=0$ or $1$).  We denote by $Q_1,\dots, Q_{\nu}\in \F_d$  the singularities of $\bar B$ distinct from $P_1, P_2$, if any.  Since $X$ is general, $R'$ intersects $B$ at two distinct points and the $Q_i$ are ordinary double points that  belong to $D$. 
We take the blow-up $\epsi\colon Y\to \F_d$ at  $P_1, P_2$ and $Q_1, \dots, Q_{\nu}$ and let $f\colon S\to Y$ be the double cover branched on the strict transform $B$ of $\bar B$. The surface $S$ is a smooth Horikawa surface of the first kind.
 We denote by $E_1, E_2$ the exceptional curves above $P_1$ and $P_2$   and by $A_1,\dots, A_{\nu}$ the exceptional curves  over $Q_1,\dots, Q_{\nu}$.  The preimages $G_1,\dots, G_{\nu}$ of $A_1, \dots, A_{\nu}$ are $(-2)$--curves,  the preimage $C$ of the strict transform $R'$ of $R$ is a $(-4)$--curve $C$ and  $X$ is the surface obtained from $S$ by contracting $C$, $G_1,\dots, G_{\nu}$.

By Proposition \ref{prop:Tlog}, there are isomorphisms $H^i(T_X)\cong H^i(T_S(-\log(C+G))$, $i=0,1,2$, where $G:=G_1+\dots+G_{\nu}$.  We are going to compute the latter groups by exploiting the double cover $f\colon S\to Y$, whose associated involution is the canonical involution $\iota$.

Set $\Delta:= R'+ A_1+\dots+A_{\nu}$; the divisors $\Delta$,  $\Delta+B$ and $f\inv(\Delta)=G+C$  are simple normal crossings. 
The following decomposition, where the first summand is $\iota$-invariant and the second one is $\iota$-antiinvariant, is a straightforward  generalization of  \cite[Prop.~4.1, a)]{rita-abel}:
\begin{gather}\label{eq:decTlog}
f_*T_S(-\log(C+G))=T_Y(-\log(B+\Delta))\oplus T_Y(-\log\Delta)(- L), 
\end{gather}
where $L=\frac12 B=\epsi^*(3D+\frac12(n+3+3d)F)-(E_1+E_2+A_1+\dots +A_{\nu})$.
To prove the theorem we compute the cohomology of the sheaves in the right hand side of \eqref{eq:decTlog}.
We will repeatedly use the following simple observation: 
\begin{lem}\label{lem:simple}
Let $\eta\colon Z\to W$ be a birational morphism of smooth surfaces and let $\mathcal F$ be a locally free sheaf on $W$. Then $\eta^*\colon H^i(\mathcal F)\to H^i(\eta^*\mathcal F)$ is an isomorphism for every $i$.
\end{lem}
\begin{proof} 
The morphism $\eta$ is a sequence of blow-ups, hence $\eta_*\OO_Z=\OO_W$, $R^i\eta_*\OO_Z=0$ for $i>0$ and the projection formula for locally free sheaves gives $\eta_*(\eta^*\mathcal F)=\mathcal F$, $R^i\eta_*(\eta^*\mathcal F)=0$ for $i>0$. So the Leray spectral sequence for $\eta^*\mathcal F$ degenerates at $E_2$ and the claim follows. 
\end{proof}

We start with a (probably well known) general lemma that describes the cotangent sheaf of a blown up surface.
\begin{lem}\label{lem:Ei}
 Let $W$ be a smooth surface, let  $\epsilon\colon \widehat W\to W$ be the blow-up at distinct points $p_1,\dots, p_k$, and let $E_1,\dots, E_k$ be the corresponding exceptional curves. Then:
\begin{enumerate}
\item there is an exact sequence $0\to \epsilon^*\Omega^1_W\to\Omega^1_{\widehat W}\to \oplus_1^k\OO_{E_i}(-2)\to 0$ \\and 
 the induced sequence \\ $0\to H^1(\epsilon^*\Omega^1_W)\to H^1(\Omega^1_{\widehat W})\stackrel{\rho}{\to} \oplus_1^kH^1(\OO_{E_i}(-2))\to 0$ is exact;
\item for every $i=1,\dots, k$ there is a nonzero $\gamma_i\in H^1(\OO_{E_i}(-2))$ such that the class $[E_i]\in H^1(\Omega^1_{\widehat W})$ maps to $(0,\dots, \gamma_i,\dots 0)\in \oplus_1^kH^1(\OO_{E_i}(-2))$.

\end{enumerate}
\end{lem}
\begin{proof} (i) 
Consider the natural   surjective map $r\colon \Omega^1_{\widehat W}\to \oplus_1^k\omega_{E_i}=\oplus_1^k\OO_{E_i}(-2)$ given by restricting 1-forms. The sheaf  $\epsilon^*\Omega^1_W$ is contained in the kernel of $r$ since the $E_i$ are contracted by $\epsi$,  and an easy computation in local analytic or \'etale coordinates shows that it actually coincides with $\ker r$. So we have the  exact sequence of sheaves in the statement.

 Taking cohomology we get an exact sequence: 
$$0\to H^1(\epsilon^*\Omega^1_W)\to H^1(\Omega^1_{\widehat W})\stackrel{\rho}{\to} \oplus_1^kH^1(\OO_{E_i}(-2)).$$
By Lemma \ref{lem:simple},  $H^1(\epsilon^*\Omega^1_W)=\epsilon^*H^1(\Omega^1_W)$ is orthogonal for the intersection form on $H^2(\widehat W,\C)$  to the subspace $V$ spanned by the classes $[E_1],\dots, [E_k]$. Since $V$ has dimension $k$, $\rho$ is surjective by dimension reasons and induces an isomorphism $V\to \oplus_1^kH^1(\OO_{E_i}(-2))$. \medskip

(ii)
For  $k=1$, the claim follows immediately from (i). Then the  general case follows easily since the points $p_1,\dots, p_k$ can be blown up in any order. 
\end{proof}

Next we prove some auxiliary results. 
We define  $H:=K_{\F_d}+3D+\frac 12(n+3+3d)F=D+ \frac 12(n+d-1)F$,  so that    $K_Y+L=\epsi^*(H)$. 

\begin{lem}\label{lem:Omega1H}
 $h^0(\epsi^*\Omega^1_{\F_d}( K_Y+L))=n-3$ and $h^i(\epsi^*\Omega^1_{\F_d}( K_Y+L))=0$ for   $i>0$.
\end{lem}
\begin{proof} 
Consider the relative differentials sequence on $\F_d$ and twist it  by $H$:
\begin{equation}\label{eq:rel-diff}
0\to \OO_{\F_d}(H-2F)\to \Omega^1_{\F_d}(H)\to \OO_{\F_d}(K_{\F_d}+H+2F)\to 0.
\end{equation}
The sheaf  $\OO_{\F_d}(K_{\F_d}+H+2F)$ has no cohomology because it has degree $-1$ on the rulings of $\F_d$, so  $h^i(\Omega^1_{\F_d}( H))=h^i(\OO_{\F_d}(H-2F))$ for all $i$.

To compute $h^i(\OO_{\F_d}(H-2F))$ we consider the projection map $p\colon \F_d\to \pp^1$ and observe that $R^ip_*\OO_{\F_d}(H-2F)=0$ for $i>0$. So 
$h^i(\OO_{\F_d}(H-2F))=h^i(p_*\OO_{\F_d}(H-2F))=h^i\left(\OO_{\pp^1}(\frac12(n+d-5)) \oplus \OO_{\pp^1}(\frac12(n-d-5)\right)$
 for all $i$. Since $\frac12(n-d-5)\ge -1$ in view of the assumption $n\ge d+3$, we get $h^i(\OO_{\F_d}(H-2F))=0$ for $i>0$ and $h^0(\OO_{\F_d}(H-2F))=n-3$.
To finish the proof recall that, by Lemma \ref{lem:simple}, $h^i(\epsi^*\Omega^1_{\F_d}(K_Y+L))=h^i(\Omega^1_{\F_d}(H))=0$ for  all $i$.
 \end{proof}

Next we deal with the anti-invariant summand $T_Y(-\log \Delta)(-L)$.
\begin{lem}\label{lem:T-}
 $h^i(T_Y(-\log \Delta) (-L))=0 $ for $i<2$, $h^2(T_Y(-\log \Delta)(-L))=n-3$.
 \end{lem}
 
\begin{proof}
By Serre duality, $H^i(T_Y(-\log \Delta)(- L))$ is the dual vector space of $H^{2-i}(\Omega^1_Y(\log\Delta)(K_Y+L))$, so we will compute the dimension of the latter group.

Since $K_Y+L=\epsi^*H$ is trivial on the curves contracted by $\epsi$, by Lemma \ref{lem:Ei} we have an exact sequence:
$$0\to \epsi^*\Omega^1_{\F_d}(K_Y+L)\to\Omega^1_Y(K_Y+L)\to \bigoplus_{i=1}^{\nu} \OO_{A_i}(-2)\oplus\OO_{E_1}(-2)\oplus\OO_{E_2}(-2) \to 0.
$$
Lemma \ref{lem:Omega1H}  gives $h^0(\Omega^1_Y(K_Y+L))=n-3$, $h^2(\Omega^1_Y(K_Y+L))=0$ and  an isomorphism  $\psi\colon H^1(\Omega^1_Y(K_Y+L))\to V_1$, where $V_1:=\bigoplus_i H^1(\OO_{A_i}(-2))\oplus H^1(\OO_{E_1}(-2))\oplus H^1(\OO_{E_2}(-2))$.
Consider  the residue sequence for $\Delta$ twisted by $K_Y+L$:
\begin{equation}\label{eq:twist}
0\to\Omega^1_Y(K_Y+L)\to \Omega_Y^1(\log \Delta)(K_Y+L)\to  \bigoplus_{i=1}^{\nu} \OO_{A_i}\oplus \OO_{R'}(1)\to 0.
\end{equation}
Set $V_0=\bigoplus_{i=1}^{\nu}H^0( \OO_{A_i})\oplus H^0( \OO_{R'}(1))$. 
We are going  to  complete the proof by showing  that  the coboundary map $\partial\colon V_0\to H^1(\Omega^1_Y(K_Y+L))$
 is an isomorphism.   
 Since the map $\psi$ is an isomorphism, it is equivalent to show instead that $\psi\circ \partial\colon  V_0\to V_1$ is an isomorphism. Finally, since $\psi\circ\partial $ is a linear map of vector spaces of the same dimension $\nu+2$, it suffices to show that $\psi\circ \partial$ is surjective.
 
Set  ${W_0}:=\bigoplus_{i=1}^{\nu}H^0( \OO_{A_i})\oplus H^0( \OO_{R'})$,  and denote by $\partial_0\colon {W_0}\to H^1(\Omega^1_Y)$  the coboundary map in the residue sequence for $\Delta$. Recall that for any component $\Gamma$ of $\Delta$ the map $\partial_0$ sends $1\in H^0(\OO_{\Gamma})$ to $[\Gamma]\in H^1(\Omega^1_Y)$.

 Given  $s\in H^0(K_Y+L)$ there is a commutative diagram 
\begin{equation}\label{diag:coboundary}
\xymatrixcolsep{5pc}
\xymatrix{
 {W_0} \ar[r]^{\partial_0 }  \ar[d]_{\cdot s}&  H^1(\Omega^1_Y)\ar[d]^{\cdot s}  \ar[r]^{\rho} & V_1 \ar[d]^{\cdot s} \\
V_0 \ar[r]^{\partial}& H^1(\Omega^1_Y(K_Y+L)) \ar[r] ^{\psi} & V_1
}
\end{equation}
where $\rho$ is the map in Lemma \ref{lem:Ei}, (i). 
The line bundle   $K_Y+L=\epsi^*(H)$ restricts to the trivial bundle on any $\epsi$-exceptional curve $\Gamma$, so the restriction to $\Gamma$ of  the multiplication by $s$ is either $0$ or an isomorphism  depending on whether $s$ vanishes on $\Gamma$ or not.
 Since $H$ is free, for $i=1,\dots, \nu$
we can  pick a section $s_i\in H^0(K+L)$ that does not vanish on $A_i$; denoting  $\alpha_i:=1\in H^0(\OO_{A_i})$, by the commutativity of diagram \eqref{diag:coboundary} we obtain $\psi(\partial (s_i\alpha_i))=s_i \rho(\partial_0(\alpha_i))= s_i\rho([A_i])$, which is a non-zero element of $H^1(\OO_{A_i}(-2))$ (cf. Lemma \ref{lem:Ei}, (ii)).  
 Since $R'$ is linearly equivalent to $\epsi^*F-E_1-E_2$, for any $s\in H^0(K_Y+L)$ one has $s[R']=-s[E_1]-s[E_2]$, since $s[\epsi^*F]\in H^1(\epsi^*(\Omega^1_Y(K_Y+L)=0$ by Lemma \ref{lem:Omega1H}.  Set $\phi:=1\in H^0(\OO_{R'})$ and choose $t_1\in H^0(K_Y+L)$ vanishing along $E_2$ and not along  $E_1$ and $t_2$ vanishing along  $E_1$ and not along $E_2$. So  $\psi(\partial(t_i\phi))=\psi(t_i\partial_0(\phi))=\psi(t_i[R'])=\psi(t_i[-E_i])=t_i\psi([-E_i])$ spans the subspace $H^1(\OO_{E_i}(-2))\subset V_1$. 
 Summing up,  $\partial$ is surjective, as claimed.
\end{proof}

We are now ready to compute the cohomology of the  invariant summand. 

\begin{lem}\label{lem:T+} 
$h^1(T_Y(-\log( \Delta+B))=7n+18-\nu$ and $h^i(T_Y(-\log( \Delta+B))=0$ for $i\ne 1$.
\end{lem}
\begin{proof}
We start by computing $h^2(T_Y(-\log( \Delta+B))=h^{0}(\Omega^1_Y(\log (\Delta+B)(K_Y))$.
We consider again the residue sequence:
\begin{equation}\label{eq:res3}
0\to \Omega^1_Y\to \Omega^1_Y(\log (B+\Delta))\to \OO_B\oplus\OO_{\Delta}\to 0.
\end{equation}
Since $H^0(\Omega^1_Y)=0$, $H^0(\Omega^1_Y(\log(B+\Delta))$ is the kernel of the coboundary map $\partial \colon H^0(\OO_B)\oplus H^0(\OO_{\Delta})\to H^1(\Omega^1_Y)$. Using the fact that a basis of  $H^1(\Omega^1_Y)$ is given  by the classes of $\epsi^*D$, $\epsi^*F$, $E_1$, $E_2$, $A_1,\dots A_{\nu}$,  one  checks  that $\partial$ is injective in all cases.
So  we conclude that $H^0(\Omega^1_Y(\log(B+\Delta))=0$.
We observe now that in  our situation $\dim|-K_Y|\ge 1$, since all blown up points of $\F_d$ belong to  the fiber $R$ or to the infinity section $D$.  Therefore $h^0(\Omega^1_Y(\log(B+\Delta)(K_Y))\le h^0(\Omega^1_Y(\log(B+\Delta))=0$

Next we consider $h^0(T_Y(-\log(B+\Delta))\le h^0(T_Y(-\log A))$, where $A$ is the unique component of $B+\Delta$ not isomorphic to $\pp^1$. We are going to show that $h^0(T_Y(-\log A))=0$.

Assume $B=A$ irreducible and consider the sequence:
$$0\to T_Y(-A)\to T_Y(-\log A)\to T_A\to 0.$$
We have $H^0(T_A)=0$, because $A$ is a smooth curve of genus larger than $0$. In addition, the isomorphism $T_Y\cong \Omega^1_Y(-K_Y)$ induces an isomorphism  $H^0(T_Y(-A))\cong H^0(\Omega^1_Y(-K_Y-A))$. The latter group is zero, since $K_Y+A$ is effective and $H^0(\Omega^1_Y)=0$. Summing up, we have proven that $\Omega^1_Y(\log (B+\Delta))(K_Y)$ has cohomology only in degree 1. Twisting \eqref{eq:res3} by $K_Y$ and taking cohomology gives:
\begin{gather} h^1( \Omega^1_Y(\log B+\Delta)(K_Y))=-\chi(\Omega^1_Y(K_Y))-\chi(K_Y|_B)-\chi(K_Y|_{\Delta})
\end{gather}
One has:
\begin{gather*} -\chi(\Omega^1_Y(K_Y))=-\chi(T_Y)=10\chi(\OO_Y)-2K_Y^2=2\nu-2, \\ -\chi(K_Y|_{\Delta})=-\chi(\OO_{R'})=-1.
\end{gather*}
If $n+3\ge 3d$, then $B$ is irreducible, $-\chi(K_Y|_B)=\chi(B|_B)=h^0(B)-1=7n+21$ and $h^1(T_Y(-\log(B+\Delta))=7n+18$. 

If $n+3=2d$, writing  $B=D'+B_0$ we get $-\chi(K_Y|_B)=\chi(D'|_{D'})+\chi(B_0|_{B_0})=-d+1+15d-1$ and $h^1(T_Y(-\log(B+\Delta))=14d-3=7n+18$. 

If $3d>n+3>2d$, we write  $B=D'+B_0$ and we get $-\chi(K_Y|_B)=\chi(D'|_{D'})+\chi(B_0|_{B_0})$.

 If $P_1\in D$, then $\chi(D'|_{D'})=-d-\nu$, $\chi(B_0|_{B_0})=3d+6n+19-\nu$
 and $h^1(T_Y(-\log(B+\Delta))=2d+6n+16=7n+18-\nu$. 
If $P_1\notin D$, then $\chi(D'|_{D'})=-d-\nu+1$, $\chi(B_0|_{B_0})=3d+6n+17-\nu$
 and $h^1(T_Y(-\log(B+\Delta))=2d+6n+15=7n+18-\nu$.
 \end{proof}

\section{Remarks and questions}

\subsection{Irreducible components of  $\mathcal M_{n}^{\rm Hor,1}$ and  $\mathcal M_{n}^{\rm Hor,2}$}

Horikawa's  papers \cite{Ho1}, \cite{Ho2}, containing the classification of surfaces of general type with $K^2=2p_g-4$ and with $K^2=2p_g-3$, predate Gieseker's  proof of the existence of the moduli  space of surfaces of general type,  hence there the author  writes of deformation types of surfaces rather than connected components of the moduli space. However he determines number and dimension  of the components and shows that they are irreducible. 

For $n\ge 6$ the proof, which is quite complicated, can be simplified as follows, exploiting the existence of the moduli space of surfaces of general type. Recall (\S \ref{sec:gen}) that for $i=1, 2$ there is a stratification $\mathcal M_{n}^{\rm Hor,i}=\sqcup_d \mathcal M_{n,d}^{\rm Hor,i}$, where $d$ ranges over a suitable subset of the integers and  $\mathcal M_{n,d}^{\rm Hor,i}$ denotes the subset of $\mathcal M_{n}^{\rm Hor,i}$ consisting of the  surfaces whose canonical image is isomorphic to $\F_d$. Using arguments similar to those in the proof of Theorem \ref{thm:TX}, but easier, one can show that for every $d$ the general surface of $\mathcal M_{n,d}^{\rm Hor,i}$ satisfies the conclusions of Theorem \ref{thm:TX}. Then the same argument as in Theorem \ref{thm:smooth} shows that the irreducible components of $\mathcal M_{n}^{\rm Hor,i}$ are the closures  of  the top dimensional strata. Finally, the components turn out to be disjoint (cf. the argument at the end of \S \ref{ssec:proof1}), hence they are actually connected components and  we have recovered Horikawa's result.

\subsection{The stratification of $\mathcal D_n$}
In Section \ref{sec:Hor-sing} we have stratified the locus $\mathcal D_{n}\subset {\mathcal M}^{\rm KSBA}_{n+1,2n-1}$ according to the value $d$ of the type  of the canonical image. 

Note that by the argument that concludes the proof of Theorem \ref{thm:Dn}  a surface in  $\mathcal D_{n,d}$ cannot be a limit of surfaces in  $\mathcal D_{n,d'}$ for $d'>d$.
By Theorem \ref{thm:smooth} for $d\ge 2$ each stratum $\mathcal D_{n,d}$ is contained in the closure of the union of the strata $\mathcal D_{n,d'}$ with $d'\le d-2$, but it is not clear what is the largest  $d'<d$ such that  $\mathcal D_{n,d}$ is contained in the closure of   $\mathcal D_{n,d'}$. Standard arguments show that $d'=d-2$  if $n\ge 3d-1$; on the other hand, looking at Table \ref{tab:Dn} we see that for $n<3d-3$ the dimension of $\mathcal D_{n,d}$ increases with $d$ so that none of these strata can be contained in the closure of the preceding one.  

One can consider  the same  question also for the analogous  stratifications of $\mathcal M_{n}^{\rm Hor,1}$ and $\mathcal M_{n}^{\rm Hor,2}$. To the best of our knowledge the answer to this question is not known in these cases either.

\subsection{Topological and differentiable structure of Horikawa surfaces}

 By a celebrated result of Freedman (\cite[Thm.~1.5]{freedman}) two simply connected 4--manifolds $M_1$, $M_2$ are  orientedly homeomorphic if and only if the intersection forms $<,>_{M_i}$ on $H^2(M_i,\Z)$ are isomorphic for $i=1,2$.  It is a classical result that a unimodular indefinite symmetric bilinear form with integer values   is determined up to isomorphism by its rank, signature and parity.

 If $S$ is a simply connected algebraic surface, then rank and signature of its intersection form can be computed from the numerical invariants (cf. \cite[\S~IX.1]{bhpv}). 
 The Horikawa surfaces of the first and second kind are simply connected, (see \cite [Thm. 3.4]{Ho1}, \cite [Thm. (4.8)]{Ho2}), so to decide whether two
Horikawa surfaces $S_1$ and $S_2$ with the same numerical invariants are homeomorphic it suffices to determine the parity of    $<,>_{S_i}$ on $H^2(S_i,\Z)$  for $i=1,2$. 

For Horikawa surfaces of the second kind $K^2$ is odd, so the intersection form is always odd.
 
The situation is a bit more complicated for Horikawa surfaces of the first  kind. 

 \begin{lem} \label{lem:1} The minimal models of the  surfaces in $\mathcal M_{n}^{\rm Hor,1a}$ have odd intersection form for all $n\ge 3$. \end{lem} 

\begin{proof} Take a double cover $ S\to \F_d$ branched on a divisor $B\in |6D+(n+3)F|$  if $n$ is odd [resp. $B\in |6D+(n+6)F|$  if $n$ is even]
such that $B$ is tangent to a  ruling $R$ of $\F_d$ at three distinct points  and is non singular otherwise.  Then the preimage  of $R$ in  $S$  splits as  the union of two $(-3)$--curves, so  the intersection form of $S$  is odd.   Since surfaces corresponding to points in the same connected  component in the Gieseker moduli space  are diffeomorphic, it follows that all surfaces in $\mathcal M_{n}^{\rm Hor,1a}$ have odd intersection form. \end{proof} 

\begin{lem}\label{lem:2} The minimal models of the  surfaces in $\mathcal M_{n}^{\rm Hor,1b}$ have even intersection form if and only if $n\equiv 5$ modulo 8. \end{lem}

\begin{proof}  If the intersection form  is even, it has signature divisible by 16 by Rochlin's Theorem (cf. \cite[Thm.IX.2.1]{bhpv}) and this condition is equivalent to  $n\equiv 5$ modulo 8. 

Conversely, if $n\equiv 5$ modulo 8 then the Hurwitz formula shows that the canonical class is divisible by 2. Since the second Stiefel--Whitney class is the mod 2 reduction of the canonical class, the intersection form is even. \end{proof} 

By Lemmata \ref {lem:1} and \ref {lem:2},  the problem of deciding whether two Horikawa surfaces with the same numerical invariants are diffeomorphic is non--trivial only when the moduli space is disconnected and there are no known  topological obstructions, namely:
\begin{itemize}
\item[(a)] for Horikawa surfaces of the first kind with $n\equiv 1$ modulo 8
\item[(b)] for Horikawa surfaces of the second kind with $n\equiv 0$ modulo 4. 
\end{itemize}
Case (a) is a long standing open problem: the question was asked for the first time by Horikawa in \cite{Ho1} and all the known differential invariants fail to distiguish the surfaces in the two connected components of the moduli space (cf. \cite{auroux}, \cite{fintushel-stern}). By Theorem 1.5 of \cite{manetti}, if the  intersection of the closures in the KSBA moduli  of two connected components $\mathcal M_1$ and $\mathcal M_2$ of the Gieseker moduli space contains a surface with T--singularities only, then the surfaces  in $\mathcal M_1$ and $\mathcal M_2$ are diffeomorphic. Our results seem to indicate that in case (b) this strategy fails for 2--Gorenstein T--singularities: by Theorem \ref{thm:smooth} the KSBA moduli space is smooth at a general point of $\mathcal D_{n,d}$ for all admissible values of $d$,  so it cannot happen that $\mathcal D_{n,d}$ is contained  in $\ol{\mathcal M^{\rm Hor,2a}_n}\cap\ol{\mathcal M^{\rm Hor,2b}_n}$.  As for case (a), in \cite{rana-rollenske} it is shown that $\ol{\mathcal M^{\rm Hor,1a}_n}\cap\ol{\mathcal M^{\rm Hor,1b}_n}$ is nonempty but not that it contains points corresponding to normal surfaces.

 We close this discussion by pointing out a surprising consequence of Theorem \ref{thm:smoothD}. For $n=4k+1$ the moduli space  $\ol{ \mathcal M^{\rm Hor,2}_n}$ is irreducible and contains two components of $\mathcal D_n$, $\mathcal D_{n,0}$ and $\mathcal D_{n,2k+2}$. If $S_1$ is the minimal resolution of a general surface  of $\mathcal D_{n,0}$ and $S_2$ is the minimal resolution of a general surface of $\mathcal D_{n,2k+2}$, then  a general surface $X\in \mathcal M^{\rm Hor,2}_n$ is diffeomorphic to  the rational blow-down  of the exceptional $(-4)$-curve of $S_1$ (cf. \cite[Example~5.9]{W}, \cite[Prop.~1.5]{manetti})  and also to the rational blow-down  of the exceptional $(-4)$-curve of $S_2$, even though, as recalled above, for $n\equiv 5$ modulo 8 the surfaces $S_1$ and $S_2$ do not have the same homotopy type. 

 On the other hand,  again by Theorem \ref{thm:smoothD},  if  $n=4k$ then $\mathcal D_n$ has two components,   $\mathcal D_{n,1}\subset \ol{\mathcal M_n^{\rm Hor, 2a}}$ and $\mathcal D_{n,2k+1}\subset \ol{\mathcal M_n^{\rm Hor, 2b}}$. The surfaces in the two components have  diffeomorphic  resolutions, since $\mathcal M_n^{\rm Hor, 1}$ is irreducible. Therefore surfaces in $\mathcal M_n^{\rm Hor, 2a}$ and $\mathcal M_n^{\rm Hor, 2b}$ are diffeomorphic to rational blow-downs of $(-4)$-curves $C_a$ and $C_b$ sitting in diffeomorphic surfaces $S_a$ and $S_b$. However this is not enough to prove that they are diffeomorphic, since we don't know whether the diffeomorphism between $S_a$ and $S_b$ sends $C_a$ to $C_b$.

\appendix
\section{Linear systems on blown up Segre--Hirzebruch surfaces}\label{sec:appendix1}

We collect here some auxiliary results  needed  in \S \ref{sec:Tchain} and in \S \ref{sec:Hor-sing}.

We recall first the following elementary fact, whose proof is left to the reader:
\begin{lem}\label {lem:P}
 Let $d\in \mathbb N_{>0}$, let $\epsi\colon Y\to \F_d$ be the blow-up at a point $P$ and let $E$ be the exceptional curve. For  $a\in \mathbb N$ set $|N_a|:=|\epsi^*(D_0+aF)-E|$.
If $a>0$ or $a=0$ and $P\notin D$, then $|N_a|$ is free. 
\end{lem}
Next  we introduce some notation. Given  $d>0$, we fix a ruling $R$ on $\F_d$  and points $P_1,P_2\in R$ with  $P_1\in \F_d$, $P_2$ possibly  infinitely near to $P_1$ and $P_2\notin D$.  Up to the action of $\Aut(\F_d)$, for $d>0$  there are four possible cases,  according to whether $P_1$ lies on $D$ or not and whether $P_2$ is infinitely near to $P_1$ or not. We denote by $\epsi_{ij}\colon Y_{ij}\to \F_d$ the blow-up at $P_1,P_2$, where $i=0$ if $P_1\notin D$ and $i=1$ otherwise and $j=0$ if $P_2$ is not infinitely near to $P_1$ and $j=1$ otherwise. For $d=0$,  we only define $Y_{00}$ and $Y_{01}$, according to whether $P_2$  is a proper  point or is infinitely near to $P_1$. We denote by $E_1$, $E_2$ the exceptional curves over $P_1$, $P_2$  ($E_1\cdot E_2=0$, $E_2$ is irreducible, $E_1$ is reducible only  when $j=1$)  and by $C'$ the strict transform on $Y_{ij}$ of a curve $C$ on $\F_d$. 

\begin{lem} \label{lem:ample} In the above setup let  $i\in\{0,1\}$, $2\le a\in \mathbb N$ and consider  on $Y_{i0}$ the line bundle $M_a:=\epsi_{i0}^*(2D_0+aF)-E_1-E_2$. 
Then   $M_a$ is nef and $R'$ is the only irreducible curve whose intersection with $M_a$ is 0. 
\end{lem}
\begin{proof}
It is immediate to check that $R'\cdot M_a=0$.  Let $\Gamma\ne R'$ be an irreducible curve such that $0=\Gamma\cdot M_a=\Gamma\cdot \epsi_{i0}^*(2D_0+(a-1)F)+\Gamma \cdot R'=0$. Then $\Gamma\cdot \epsi_{i0}^*(2D_0+(a-1)F)=\Gamma \cdot R'=0$ because $\epsi_{i0}^*(2D_0+(a-1)F)$ is nef.   Since  $E_i\cdot R'=1$ for $i=1,2$, the curve $\Gamma$ does not coincide with either $E_1$ or $E_2$, hence it   is not contracted by $\epsi_{i0}$, so it maps to a curve $\bar \Gamma$ of $\F_d$ such that $\bar\Gamma\cdot (2D_0+(a-1)F)=0$. This is a contradiction because $2D_0+(a-1)F$ is ample. 
\end{proof}
   
\begin{lem} \label{lem:tangent} In the above setup let  $0<d\in \mathbb N$,  $i\in\{0,1\}$, $a\in \mathbb N$ and consider  on $Y_{i1}$ the linear system $|M_a|:=|\epsi_{i1}^*(2D_0+aF)-E_1-E_2|$. 

If either  $a>0$ or $a=0$ and $i=0$ then   $|M_a|$ is base point free. 
\end{lem}
\begin{proof}
The restriction map $H^0(2D_0+aF)\to H^0(\OO_R(2))$ on $\F_d$  is surjective since $h^1(2D_0+(a-1)F)=0$, hence we may find $C\in |2D_0+aF|$ such that $C|_R=2P_1$. So $\epsi_{i1}^*C-E_1-E_2$ is disjoint from $R'$ and therefore no base point of $|M_a|$ lies on $R'$. 

If $a>0$, then the base locus of     the subsystem   $R'+|\epsi_{i1}^*(2D_0+(a-1)F)|\subset |M_a|$ is $R'$, so $|M_a|$ is free.

 If $a=0$  then the base locus of   the subsystem $R'+D'+|\epsi_{i1}^*(D_0+(d-1)F)|\subset |M_0|$ is $R'+D'$, so the base points of $|M_0|$, if any, are contained in $D'$. However, we may find a curve $C\in |D_0|$ going through $P_1$ and disjoint from $D$. Then $|M_0|$ is free since it   contains the curve  $2C'+(E_1-E_2)$ that is disjoint from $D'$.\end{proof}

\begin{prop}\label{prop:P1P2} In the above set-up 
let $n,d\in \mathbb N$ such that   $n\ge \max\{2d-3, d+3\}$ and $n-d$ is  odd and let $i,j\in\{0,1\}$.\footnote{If $d=0$, then $i=0$, per our conventions.}
  Set $L_{ij}:=\epsi_{ij}^*(6D+(n+3+3d)F)-2E_1-2E_2|$   and write $|L_{ij}|=Z_{ij}+|M_{ij}|$, where $Z_{ij}$ is the fixed part and $|M_{ij}|$ is the moving part.

Then:
\begin{enumerate}
\item  the general curve of $|L_{ij}|$ is non--reduced if and only if $n=2d-3$ and $i=1$ or $n=2d-2$ and $i=j=1$. 
\item if the general curve of $|L_{ij}|$  is reduced then  $|M_{ij}|$ is free, $M_{ij}^2>0$, the general curve in $|M_{ij}|$ is smooth
 and $Z_{ij}$ is either zero or snc. 
\item  $\dim|L_{i0}|=\dim|L_{i1}|$ for $i=0,1$
\end{enumerate}
In addition, $\dim |L_{00}|$,  $\dim |L_{10}|$ and $Z_{ij}$ can be read off  Table \ref{tab:M} and Table \ref{tab:Z}. 
\footnote{The  symbol \xmark\,  means that either the linear system is not defined or falls in case (i)}

\begin{table}[h!]  
\caption{}\label{tab:M}
  \begin{center}
    {\rm
       \begin{tabular}{|c|c|c|} 
      \hline
       & $\dim|L_{00}|$ &  $\dim|L_{10}|$ \\
       \hline
       $n\ge 3d-1$  &$7n+21$ &   $7n+21$  \\
      \hline
     $ n=3d-3$   & $7n+21$ &  $7n+22$ \\
       \hline
     $ 2d-1\le n<3d-3$  & $6n+3d+17$   & $6n+3d+19$ \\
      \hline
      $ n=2d-2 $  & $7n+d+19$   & $7n+d+21$\\
      \hline
         $n=2d-3$  & $6n+3d+17$  & \xmark  \\
      \hline
    \end{tabular}}
  \end{center}
\end{table}

\begin{table}[h!]
\caption{}
 \label{tab:Z}
  \begin{center}
    {\rm
        \begin{tabular}{|c|c|c|c|c|}       \hline
      &$Z_{00}$ & $Z_{01}$ & $Z_{10}$& $Z_{11}$ \\
       \hline
       $n\ge 3d-1$  &$0$ & $0$&  $0$& $0$  \\
      \hline
      $ n=3d-3$   & 0 & 0 & $D'$  &$D'+(E_1-E_2)$\\
       \hline
      $ 2d-1\le n<3d-3$  & $D'$ & $D'$ & $D'$  &$D'+(E_1-E_2)$ \\
      \hline
       $n=2d-2$  & $ D' $  &$D'$   & $D'$ & \xmark  \\
      \hline

       $n=2d-3$  & $ D' $  &$D'$   & \xmark & \xmark  \\
      \hline
    \end{tabular}}
  \end{center}
\end{table}
\end{prop}
\begin{proof}
Consider first the case  $n= 2d-3$. We have  $|L_{1j}|=D'+|\epsi_{1j}^*(5D_0)-E_1-2E_2|$ and $D'\cdot(\epsi_{1j}^*(5D_0)-E_1-2E_2)=-1<0$, so $D'$ appears  in the fixed part of $|L_{1j}|$ with multiplicity larger than $1$.
Similarly, if  $n=2d-2$  then   $|L_{11}|=D'+(E_1-E_2)+| \epsi_{11}^*(5D_0+F)-2E_1-E_2|$ and $D'\cdot(\epsi_{11}^*(5D_0)+F-2E_1-E_2)=-1<0$, so again $D'$  appears  in the fixed part of $|L_{11}|$ with multiplicity larger than $1$. 
So in the above   cases the general curve of $|L_{ij}|$ is not reduced, as claimed. 

To complete the proof we are going to  determine $Z_{ij}$ and $|M_{ij}|$ in all the remaining cases. In each case it turns out that   that $|M_{ij}|$ is a free system with $M_{ij}^2>0$. So $M_{ij}$ is nef  and big and a general divisor  $\Delta\in |M_{ij}|$   is smooth and irreducible. Since $h^1(\OO_Y)=0$, we get $\dim |L_{ij}|=\dim |M_{ij}|=h^0(M_{ij}|_{\Delta})$. In addition  Riemann-Roch theorem on $\Delta$ gives $h^0(M_{ij}|_{\Delta})=\frac 12(M_{ij}^2-M_{ij}\cdot K_{Y_{ij}})$,   because $K_{Y_{ij}}\cdot M_{ij}=(-\epsi_{ij}^*(D+D_0+F)-R')\cdot M_{ij}\le -\epsi_{ij}^*D_0\cdot M_{ij}<0$. So Table \ref{tab:M} will  be filled automatically once  we have completed  Table \ref{tab:Z}. 

If  $n< 3d-3$,  then $D'\le Z_{ij}$ for $i,j\in\{0,1\}$, since $D$ is in the fixed part of $|6D+(n+3+3d)F|$ on $\F_d$.
The system  $|\epsi_{0j}^*(5D_0)-2E_1-2E_2|$ is easily seen to be free for $j=0,1$. In fact this is immediate for $j=0$ and for $j=1$ this system contains the divisors of  $|2(\epsi_{01}^*(2D_0)-E_1-E_2)|+|\epsi_{01}^*D_0|$, hence  it is free by Lemma \ref{lem:tangent}. So for $2d-3\le n<3d-3$ one has  $Z_{0j}=D'$ and $|M_{0j}|=|\epsi^*_{0j}(5D_0+(n+3-2d)F)-2E_1-2E_2|$ is free. 

Now look at  $|\epsi_{10}^*(5D_0+F)-E_1-2E_2|$: this system contains $|\epsi_{10}^*(D_0+F)-E_1|+|2\epsi_{10}^*(2D_0)-E_2|$ and therefore is free.  So for $2d-2\le n<3d-3$ one has  $Z_{10}=D'$ and $|M_{10}|=|\epsi^*_{10}(5D_0+(n+3-2d)F)-E_1-2E_2|$ is free. 

 Last consider  $i=j=1$: $(E_1-E_2)$ is a $(-2)$--curve and $(E_1-E_2)\cdot (L_{11}-D')=-1$, so $D'+(E_2-E_1)\le Z_{11}$. The system $|\epsi_{11}^*(5D_0+2F)-2E_1-E_2|$ contains $|\epsi_{11}^*(2D_0+F)-E_1-E_2|+|\epsi_{11}^*(3D_0+F)-E_1|$, so it is free by Lemma \ref{lem:tangent} and Lemma \ref{lem:P}. It follows that for  $2d-1\le n<3d-3$ one has  $Z_{11}=D'+(E_1-E_2)$ and $|M_{11}|=|\epsi^*_{11}(5D_0+(n+3-2d)F)-2E_1-E_2|$ is free.  
 
Assume  $n= 3d-1$. 
The system $|L_{i0}|$ is free for $i=0,1$, by Lemma \ref{lem:P}, since it contains $|2(\epsi_{i0}^*(D_0)-E_2)|+|2(\epsi_{i0}^*(D_0+F)-E_1))|+|2\epsi_{i0}^*(D_0)|$. If $j=1$, we observe instead that for $i=0,1$ the system $|L_{i1}|$ contains $|2(\epsi_{i1}^*(2D_0+F)-E_1-E_2)|+|2\epsi_{i1}^*D_0|$, and so it is free by Lemma \ref{lem:tangent}. Therefore $|L_{ij}|$ is free for $i,j\in\{0,1\}$ and $n=3d-1$. A fortiori, $|L_{ij}|$ is free for $i,j\in\{0,1\}$ and $n\ge 3d-1$. 

The case $n=3d-2$ is ruled out by the assumption that $n-d$ is odd, so we are left with the ``borderline case''  $n=3d-3$.   
The system $|L_{0j}|$ is easily seen to be free  for $j=0,1$ arguing as for $3d-1$.
For $i=1$ one has $D'\cdot L_{1j}=-2$ and therefore $D'$ is a fixed component  of $|L_{1j}|$.  Again, arguing as above it is not hard to check that  $|L_{10}-D'|=|\epsi_{10}^*(5D_0+dF)-E_1-2E_2|$ is free. 

If $P_2$ is infinitely near to $P_1$, then $E_1-E_2$ is a $(-2)$--curve and $(E_1-E_2)\cdot (L_{11}-D')=-1$, therefore $D'+(E_1-E_2)\le Z_{11}$. The condition $3d-3=n\ge 3$ implies $d\ge 2$, so  $|L_{11}-D'-(E_1-E_2)|=|\epsi_{11}^*(5D_0+dF)-2E_1-E_2|$ can be shown to be free as in case $2d-1\le n<3d-3$. The proof is now complete. 
\end{proof}

 \end{document}